\documentclass[12pt]{article}
\usepackage{amsmath,amssymb,amsthm,mathtools,mathrsfs,latexsym}
\usepackage{graphicx}
\usepackage{xcolor}
\usepackage{tikz}
\usepackage[linkcolor=blue,colorlinks=true,urlcolor=red,bookmarksopen=true]{hyperref}
\usepackage{newclude}
\usepackage{imakeidx}
\makeindex
\usepackage{indentfirst}
\usepackage{bm}

\allowdisplaybreaks

\title{Global regularity of the multi-dimensional compressible Navier-Stokes-Korteweg system with large initial data}
\date{}
\author{
	\bf\large Xiangdi Huang$^{a}$\thanks{E-mail addresses: xdhuang@amss.ac.cn (X. Huang); mengweili@amss.ac.cn (W. Meng); xyzhang05@163.com (X. Zhang). }, Weili Meng$^a$, Xueyao Zhang$^b$\\
	\small a. State Key Laboratory of Mathematical Sciences, Academy of Mathematics and Systems Science,\\
	\small Chinese Academy of Sciences, Beijing 100190, China;\\
	\small b. School of Mathematics and Statistics, Yulin University, Yulin 719000, China;\\
}

\newcommand{\divg}{{\rm  div}}
\newcommand{\norm}[1]{\left\Vert#1\right\Vert}
\let\div\relax
\DeclareMathOperator*{\div}{div}

\usepackage{hyperref}
\hypersetup{hypertex=true,
	colorlinks=true,
	linkcolor=blue,
	anchorcolor=blue,
	citecolor=blue}

\usepackage{color}
\newtheorem{thm}{Theorem}[section]

\newtheorem{lema}{Lemma}[section]
\newtheorem{prop}{Proposition}[section]

\newtheorem{rmk}{Remark}[section]

\allowdisplaybreaks

\begin{document}
	\maketitle %显示标题
	\begin{abstract}
In this paper, we establish the global existence of strong solutions to the multi-dimensional compressible Navier–Stokes–Korteweg system with arbitrarily large initial data on the torus $\mathbb{T}^N$, provided the adiabatic exponent satisfies $\gamma \in [1, \infty)$ when $N = 2$ or $\gamma \in [1, \frac{7}{3})$ when $N = 3$. This system was derived by Dunn and Serrin [Arch. Ration. Mech. Anal. 88(2):95–133, 1985] and is widely used to model capillarity in compressible fluids. Via an original modified Nash-Moser type iteration, we establish a critical estimate linking the effective velocity $u+\nabla\log \varrho$ and $\varrho^{-1}$, namely
$$\norm{u+\nabla\log\varrho}_{L^\infty(\mathbb{T}^N\times(0,T))}\le C\sqrt{\log(e+\norm{\varrho^{-1}}_{L^\infty(\mathbb{T}^N\times(0,T))})},$$ which plays a crucial role in deriving the positive lower bound of the density. To our knowledge, this can be viewed as the  first existence result of global strong solutions for the compressible fluid dynamics equations with physical significance in general three-dimensional domains with arbitrarily large initial data. \\[4mm]
	\end{abstract}

	\section{Introduction}
    In order to provide a continuum mechanical model for capillarity within fluids, in 1901, Korteweg \cite{Korteweg-1901} formulated a constitutive equation for the Cauchy stress tensor which includes terms depending on density gradients. Dunn and Serrin \cite{DuSe-1985} further modified the system of compressible fluids based on the Korteweg theory of capillarity. The compressible Navier–Stokes–Korteweg (NSK) system with density-dependent viscosity coefficients takes the form:
	\begin{equation}\label{Equ1}
		\begin{cases}
			\varrho_t + \operatorname{div}(\varrho u) = 0,\\
		(\varrho u)_t + \operatorname{div}(\varrho u\otimes u)+ \nabla P(\varrho ) - \operatorname{div}\big(2\mu(\varrho) \mathbb{D}u \big) - \nabla \big(\lambda(\varrho) \operatorname{div} u\big)
			= \divg\mathbb{K},
		\end{cases}
	\end{equation}
	where $\varrho(x,t)$, $u(x,t)$ and $P(\varrho)=\varrho^\gamma$ are the density,
	velocity, and pressure of the fluid, respectively, and $\mathbb{D}u = \frac{{\nabla u + {\nabla u}^T}}{2}$. The viscosities $\mu(\varrho)$ and $\lambda(\varrho)$ satisfy
	$$\mu(\varrho)\geq0,\ \mu(\varrho)+N\lambda(\varrho)\geq0,$$
    where $N$ denotes the spatial dimension. The capillary term is given by
		\begin{align*}
		\mathbb{K}=\Big(\varrho\divg(\kappa(\varrho)\nabla\varrho)-\frac{\varrho\kappa'(\varrho)-\kappa(\varrho)}{2}|\nabla\varrho|^2\Big)\mathbb{I}
        -\kappa(\varrho)\nabla\varrho\otimes\nabla\varrho,
		\end{align*}
        where $\mathbb{I}$ is the identity matrix. Note that when $\kappa(\varrho)=0$, the system 
(\ref{Equ1}) is reduced to the compressible Navier–Stokes system.

In this paper, we require the coefficients in the system \eqref{Equ1} to satisfy
		\begin{align}\label{vis coeff}
			\mu(\varrho)=\nu_0\varrho,\quad\lambda(\varrho)=0,\quad\kappa(\varrho)=\frac{\kappa_0^2}{\varrho},
		\end{align}
where the positive constants $\nu_0$ and $\kappa_0$ satisfy
\begin{align}\label{vis'}
    \nu_0=\kappa_0>0.
\end{align}
Without loss of generality, we always assume that \(\nu_0 = \kappa_0 = 1\).

When the viscosity and capillarity coefficients in system \eqref{Equ1} satisfy condition \eqref{vis coeff}, the compressible Navier-Stokes-Korteweg system reduces to the quantum Navier-Stokes system. For a more detailed derivation of the quantum Navier–Stokes system, we refer to \cite{BrMe-2010,Jun-2012}. Similar systems also play a central role in quantum fluid dynamics. A prominent example is the inviscid case ($\nu_0=0$), which reduces to the well-known quantum hydrodynamics model for superfluids \cite{LaLi-1977}. These models are employed in the study of quantum semiconductors \cite{FeZh-1993}, weakly interacting Bose gases \cite{Gr-1973}, and Bohmian trajectories \cite{Wy-2005}. In recent decades, these models have attracted considerable attention, driven in part by advances in nanotechnology. For the one-dimensional quantum Navier--Stokes system, J\"{u}ngel \cite{Jun-2011} established the global existence of smooth solutions with strictly positive particle densities on $\mathbb{R}$ for arbitrarily large initial data, under the assumption $\nu_0 = \kappa_0$. Subsequently, Chen and Zhao \cite{ChenZhao-2025} extended this result to the case $\nu_0 \geq \kappa_0$. For the high-dimensional quantum Navier--Stokes system, Jüngel \cite{Jun-2010} obtained the global weak solutions to the system with $\nu_0<\kappa_0$ on the torus for large initial data, by employing the special test function $\rho \phi$
with $\phi$ smooth and compactly supported. By the construction of a regular approximating system, Antonel and Spirito \cite{AnSp-2017,AnSp-2018}  proved the global existence of finite energy
weak solutions on the two- and three-dimensional torus for large initial data with $\nu_0>\kappa_0$ and other technical restrictions imposed on $\nu_0$ and $\kappa_0$, where the vacuum region is included in the weak formulations. Recently, Lacroix-Violet and Vasseur \cite{LaVa-2018} proved the existence of the renormalized
solution which implies that the above technical restriction in \cite{AnSp-2017} can indeed be removed, and the semi-classical limit
of the solution is studied. The global existence of finite energy weak solutions in the
whole space with non trivial far-field condition is proved by Antonelli, Hientzsch and Spirito \cite{AnHiSp-2021}. Some results about global weak solution of the quantum Navier-Stokes equations with damping term are considered in Vasseur and Yu \cite{VaYu-2016} and Lü, Zhang and Zhong \cite{LvZh-2019} and the references therein.
Although the study of weak solutions to the high‑dimensional quantum Navier–Stokes system is well developed, the existence of global strong solutions for large initial data in higher dimensions remains open. This is precisely the objective of the present paper.

In what follows, we shall focus on some closely related works about the well-posedness of the broader Navier-Stokes-Korteweg system.  The investigation of Navier-Stokes equations incorporating a capillarity term has garnered significant
attention in recent research. Serrin \cite{Serrin-1983} reconsidered the Korteweg theory and showed the existence of steady profile
connecting two different phases. Slemrod \cite{Slemrod-1983} considered the existence of
travelling wave solutions connecting two different phases. In two or three space dimensions, for initial data away from vacuum, Hattori and Li \cite{HaLi-1994} proved the local existence of strong solution of the Cauchy problem for the compressible fluid models of Korteweg type in Sobolev space. Later, Kotschote \cite{Kots-2008} considered the initial-boundary-value problem and proved the local existence of the strong solution. In subsequent research, there have been many mathematical results on the global existence of compressible fluid models of Korteweg type. 

In studies of the one-dimensional space case, Tsyganov \cite{Tsy-2008} proved the global existence and asymptotic behaviors of weak solutions in bounded domain with large data away from vacuum. For the system (\ref{Equ1}) with density-dependent viscosity and capillary coefficients, Germain and LeFloch \cite{GerLef-2016} established the global existence of finite energy weak solutions of the Cauchy problem as well as their convergence toward the entropy solutions to the Euler system. Burtea and Haspot \cite{BurHas-2022} studied the capillarity vanishing limit scenarios. Specifically, when the viscosity $\mu(\varrho)=\varrho^\alpha$ and the capillarity coefficient $\kappa(\varrho)=\varrho^\beta$ satisfy some given conditions, Antonelli, Bresch and Spirito \cite{AnBrSp-arxiv} obtained the global weak solutions of periodic problem with large data.
For existence results concerning global strong solutions, we refer to \cite{ChenCai-2015,Chen-2016,ChenLi-2021,BurHas-20222,Dong-2026}.

However, relatively few results are available for high-dimensional problems. Hattori and Li \cite{HaLi-1996} extended the local results obtained in \cite{HaLi-1994}
 and established the global existence of solutions for the Korteweg system with the small and non-vacuum initial data of the Cauchy problem, where the viscosity and capillary coefficients are constant.  When the viscosity and capillarity coefficients depend on the density, in a functional setting as close as possible to the physical energy spaces, Danchin and Desjardins \cite{DaDe-2001} proved the global existence of suitably smooth solutions which are close to a stable non-vacuum equilibrium. Bresch, Desjardins, and Lin \cite{BDL-2003} proved both the $L^1$ stability and the global existence of weak solutions for Korteweg-type fluids in a periodic box or strip domain, allowing for vacuum. In their analysis, a novel a priori entropy estimate was established, yielding higher regularity for the density. Bresch and Desjardins \cite{BD-2003} considered the two dimensional viscous shallow water model with a capillary term allowing vacuum, and obtained the global existence of weak solutions which converge to the strong solution of the viscous quasi-geostrophic equation with free surface term.  Under a condition of smallness on the initial data, Haspot \cite{Haspot-2016} investigated
the existence of local and global strong solutions with some different types of density-dependent viscosity and capillarity coefficients. 

In sharp contrast to the one-dimensional case, the global existence of strong solutions to the multi‑dimensional Navier–Stokes–Korteweg system is currently known only under smallness conditions on the initial data. In fact, the problem of global existence for arbitrarily large data remains largely open in higher dimensions. As a first step toward this goal, we establish the global existence of strong solutions to system \eqref{Equ1}–\eqref{vis'} on the torus $\mathbb{T}^N$ with large initial data away from vacuum.

		More precisely, we study system \eqref{Equ1}--\eqref{vis'} with prescribed initial data \((\varrho_0,u_0)\), which are periodic with period $1$ in each space direction. We require that
        \begin{align}\label{ini data}
			\varrho(x,0)=\varrho_0(x),\quad u(x,0)=u_0(x),\quad x\in\mathbb{T}^N,
		\end{align}
where \(\mathbb{T}^N=\mathbb{R}^N/\mathbb{Z}^N\) with \(N=2,3\). With the choice $\kappa(\varrho)=\frac{1}{\varrho}$ from \eqref{vis coeff} and \eqref{vis'}, the Korteweg tensor takes the following form:
		\begin{align*}
			\divg\mathbb{K}=\divg(\varrho\nabla\nabla\log\varrho).
		\end{align*}
		When the density is positive, we define the effective velocity, introduced by J\"ungel \cite{Jun-2010}, as
		\begin{align*}
			v=u+\nabla\log\varrho.
		\end{align*}
		This transformation converts system \eqref{Equ1}--\eqref{vis'} into the following parabolic system:
		\begin{equation}
			\label{Equ2}
			\left\{
			\begin{array}{l}
				\varrho_t+\div(\varrho v)-\Delta\varrho=0,\\
				\varrho v_t+\varrho(v-\nabla\log\varrho)\cdot\nabla v+\nabla P=\div(\varrho\nabla v).
			\end{array}
			\right.
		\end{equation}
		Accordingly, the initial effective velocity $v_0$ is given by
		\begin{align}\label{ini v_0}
			v_0=u_0+\nabla\log\varrho_0.
		\end{align}
        A key structural feature of \eqref{Equ2}$_2$ relies on a precise relation between the coefficients in \eqref{vis coeff}, that is, $\nu_0 = \kappa_0$.
		
		Before stating our main results, we introduce the following notations and conventions throughout this paper:
		\begin{align*}
			\int fdx=\int_{\mathbb{T}^N} fdx,\quad \int_{Q_T}f dxdt=\int_0^T\int f dxdt,\quad Q_T=\mathbb{T}^N\times[0,T].
		\end{align*}
		For \(1\le s\le\infty\), we denote the standard Lebesgue and Sobolev spaces as follows:
		\begin{align*}
			L^s=L^s(\mathbb{T}^N),\quad W^{k,s}=W^{k,s}(\mathbb{T}^N),\quad H^s=W^{s,2}.
		\end{align*}
		For \(1\le a,s\le\infty\), \(k\in\mathbb N^+\) and \(T>0\), we denote
		\begin{align*}
        \norm{f}_{L^a_TL^s}=\norm{f}_{L^a(0,T;L^s)},\quad\norm{f}_{L^a_TW^{k,s}}=\norm{f}_{L^a(0,T;W^{k,s})}.
		\end{align*}
        The constants $C$ appearing in the proofs of Sections 2--4 may vary from line to line and generally depend on quantities such as certain norms of the initial data or the system parameters stated in the corresponding proposition or lemma.
        
		We now state the main result on the global existence of strong solutions to problem \eqref{Equ1}--\eqref{ini data}. 
		\begin{thm}\label{Thm 1.1}
			Let $N=2$ or $N=3$. Assume that $\gamma$ satisfies
			\begin{align}
				&N=2,\quad \gamma\in[1,\infty);\label{2d gamma}\\
				&N=3,\quad \gamma\in\Big[1,\frac{7}{3}\Big),\label{3d gamma}
			\end{align}
			and that the initial data $(\varrho_0,u_0)$ satisfy 
			\begin{align}
				0<\underline{\varrho_0}\le \varrho_0\leq \overline{\varrho_0},\quad\varrho_0\in H^3,\quad u_0\in H^2,
			\end{align}
			where $\underline{\varrho_0}$ and $\overline{\varrho_0}$ are positive constants. Then the problem \eqref{Equ1}--\eqref{ini data} admits a unique global strong solution $(\varrho,u)$ satisfying for any $0<T<\infty$ and $(x,t)\in \mathbb{T}^N\times[0,T]$,
			\begin{align}
				\left\{
				\begin{array}{l}
					(C(T))^{-1}\leq \varrho(x,t)\leq C(T),\\
					\varrho\in C([0,T];H^3)\cap L^2(0,T;H^4),\ \varrho_t\in C([0,T];H^1)\cap L^2(0,T;H^2),\\
					u\in C([0,T];H^2)\cap L^2(0,T;H^3),\ u_t\in L^\infty(0,T;L^2)\cap L^2(0,T;H^1),
				\end{array}
				\right.
			\end{align}
		where the constant \(C(T)>0\) depends on the initial data and \(T\).
		\end{thm}
		\begin{rmk}\label{rmk1}
		    Haspot \cite{Haspot-2017} studied the existence of strong solutions to the Cauchy problem for the isothermal case \((\gamma=1)\) of the same compressible Navier-Stokes-Korteweg system \eqref{Equ1}--\eqref{vis'}. In his paper, he introduced a series of highly original and important ideas and techniques. Indeed, there is a crucial estimate in Haspot's article that requires the estimate for $\norm{\varrho^{\frac{1}{p+2}}v}_{L^\infty_TL^{p+2}}$ independent of $p$, which plays a key role in deriving the positive lower bound of the density. Unfortunately, following his approach, we are unable to obtain such a uniform estimate---even for the isothermal case of \cite{Haspot-2017} discussed in his paper. In fact, the estimate for $\norm{\varrho^{\frac{1}{p+2}}v}_{L^\infty_TL^{p+2}}$ necessarily depends on $p$. This dependence prevents us from deriving an upper bound for the effective velocity by passing to the limit $p \to \infty$, as was done in \cite{Haspot-2017},  ultimately leading to the inability to obtain a positive lower bound estimate for the density. To overcome this difficulty, we prove another novel critical control relation between $v$ and $\varrho^{-1}$, i.e., $\norm{v}_{L^\infty(Q_T)}\le C\sqrt{\log(e+\norm{\varrho^{-1}}_{L^\infty(Q_T)})}$. Based on these new observations, we are able to prove the results for the global strong solutions with large initial data for the isothermal and non-isothermal cases. One can refer to Proposition \ref{Prop v inf} for the details. Therefore, Theorem \ref{Thm 1.1} can be regarded as presenting the first existence result of global strong solutions for the compressible fluid dynamics equations with physical significance, applicable to the three-dimensional general domains and large initial data. In contrast, for incompressible fluids, the only known existence result of global strong solutions for three-dimensional general domains with large initial data dates back to Cao and Titi \cite{Cao-Titi} in 2007.
		\end{rmk}
		\begin{rmk}
       Another crucial argument is that we can establish the upper bound of the density to the general adiabatic exponent $\gamma$ in the three-dimensional case. One can refer to Proposition \ref{Prop R_T} for the details. 
        \end{rmk}
		\begin{rmk}
    We primarily focus on the case $\gamma>1$ and $N=3$. For the isothermal case with $\gamma=1$, its proof can be carried out almost similarly to the non-isothermal case line by line. For the case $N=2$, we give a brief discussion in the final section.
        \end{rmk}

        The proof of Theorem \ref{Thm 1.1} mainly consists of three steps. In the first step, we obtain an upper bound for the density via the mass equation, which is primarily achieved through the high integrability of the effective velocity and De Giorgi iteration. In the second step, we establish a precise control relation between the effective velocity and the lower bound of the density, as stated in Remark \ref{rmk1}. This estimate, in turn, yields a positive lower bound for the density itself. To this end, we devise an original modified Nash–Moser iteration scheme to estimate the $L^\infty$-norm of the effective velocity and that of $\varrho^{-1}$ step by step. Notably, the method developed here can be adapted to enhance the regularity of solutions for a broad class of parabolic systems, from finite $L^p$ regularity to $L^\infty$ bounds. In the third step, based on the estimates for the upper and lower bounds of the density, we obtain the higher-order regularity of both the density and the effective velocity, which ensures that the strong solution exists globally.
        
        The remainder of the paper is organized as follows. Section 2 establishes the upper bound for the density, while Section 3 provides the positive lower bound for the density. Section 4 is devoted to higher-order derivative estimates for the density and the effective velocity. Finally, the proof of Theorem~\ref{Thm 1.1} is completed in Section 5.

	\section{Upper bound for the density (3D)}
	Owing to the better parabolic structure of \eqref{Equ2} compared with the original system \eqref{Equ1}, we focus on \eqref{Equ2}. Standard theory yields that system \eqref{Equ2}--\eqref{ini v_0} admits a unique local strong solution $(\varrho,v)$ on $\mathbb{T}^{3}\times[0,T^*)$ with maximal existence time \(T^* > 0\), and this solution satisfies $\varrho>0$ on its interval of existence. To prove that the solution is global, we argue by contradiction. Suppose that \(T^* < \infty\). Then, for any \(T\) satisfying \(0 < T < T^{*}\), the solution is well-defined on \(\mathbb{T}^3\times[0, T]\).
	
	We begin with the standard energy estimate.  Denote 
	\begin{align*}
			E_0=\int \left(\frac{1}{2}\varrho_0|u_0|^2+\frac{1}{\gamma-1}\varrho_0^\gamma+2|\nabla\sqrt{\varrho_0}|^2\right)dx.
	\end{align*}
	\begin{prop}\label{Prop ene}
	There exists a constant \(C>0\) depending on \(\gamma\) and \(E_0\) such that
		\begin{align}\label{ene est}
			\sup_{0\leq t\leq T}\int \left(\varrho|u|^2+\varrho^\gamma+|\nabla\sqrt{\varrho}|^2\right)dx+\int_{Q_T} \varrho|\mathbb{D} u|^2dxdt\leq C.
		\end{align}
	\end{prop}
	For brevity we omit the proof of Proposition \ref{Prop ene}; details can be found in Proposition 3.1 of \cite{Yu-Wu-2021}. The standard approach then gives the following $L^2$-estimate for $v$.
	\begin{prop}
		There exists a constant \(C>0\) depending on \(\gamma\) and \(E_0\) such that
		\begin{align}\label{ent est}
			\sup_{0\leq t\leq T}\int \varrho|v|^2dx+\int_{Q_T} |\nabla\varrho^{\frac{\gamma}{2}}|dxdt+\int_{Q_T}\varrho|\nabla v|^2dxdt\leq C.
		\end{align}
	\end{prop}
	
	The following Proposition \ref{Prop high v} shows the higher integrability of the effective velocity, which is crucial for obtaining an upper bound on the density via De Giorgi iteration. For any \(\gamma\in(1,\tfrac73)\) we choose \(q\in(1,4)\) depending only on \(\gamma\) such that
	 \begin{align}\label{u1}
		\gamma<1+\frac{4}{q+2}.
	\end{align}
	 \begin{prop}\label{Prop high v}
		There exists a constant \(C>0\) depending on \(T^*, q, \gamma, E_0,\) and \(  \|\varrho_0^{1/(q+2)}v_0\|_{L^{q+2}}\) such that
		\begin{align}\label{u1.5}
			\sup_{0\leq t\leq T}\norm{\varrho^{\frac{1}{q+2}}v}_{L^{q+2}}\leq C.
		\end{align}
	\end{prop}
	\begin{proof}
		Multiplying equation \(\eqref{Equ2}_2\) by \(|v|^{q}v\) and integrating the resulting equation over \(\mathbb{T}^3\), we use integration by parts and Young's inequality to obtain
			\begin{align*}
			\begin{split}
				&\quad\frac{1}{q+2}\frac{d}{dt}\int \varrho |v|^{q+2}dx+\int \varrho |v|^q|\nabla v|^2dx+q\int \varrho |v|^q|\nabla |v||^2dx\\
				&=\int \varrho^\gamma\divg(|v|^qv)dx\\
				&=\int \varrho^\gamma|v|^q\div vdx+ q\int \varrho^\gamma|v|^{q-1}v\cdot\nabla |v|dx\\
				&\leq \frac{1}{2}\int \varrho |v|^q|\nabla v|^2dx+\frac{q}{2}\int \varrho |v|^q|\nabla |v||^2dx+C\int \varrho^{2\gamma-1}|v|^qdx,
			\end{split}
		\end{align*}
		where we have used the simple fact $|\divg v|\leq \sqrt{3}|\nabla v|$. H\"{o}lder's inequality implies that
		 	\begin{align}\label{u1.7}
			\begin{split}
				\frac{d}{dt}\int \varrho |v|^{q+2}dx&\leq C\int \varrho^{2\gamma-1}|v|^qdx\\
				&\leq C\left(\int \varrho|v|^{q+2}dx\right)^{\frac{q}{q+2}}\left(\int \varrho^{q\gamma+2\gamma-q-1}dx\right)^{\frac{2}{q+2}}.
			\end{split}
		\end{align}
		To handle the right-hand side of the above inequality, we denote
		\begin{align*}
			\zeta=\frac{q\gamma+2\gamma-q+4}{2},\quad\xi=\frac{4-q}{q\gamma+2\gamma-q+4}.
		\end{align*}
		Using \eqref{ene est} and \eqref{ent est}, we obtain from the standard Sobolev embedding and H\"{o}lder's inequality that
			\begin{align}\label{u2}
			\begin{split}
				\norm{\varrho}_{L_T^{\frac{2}{q+2}\zeta}L^\zeta}&\leq C\norm{\varrho}_{L^\infty_T L^3}^{\xi}\norm{\varrho}_{L^\gamma_T L^{3\gamma}}^{1-\xi}\\
				&= C\norm{\sqrt{\varrho}}_{L^\infty_TL^6}^{2\xi}\norm{\varrho^{\frac{\gamma}{2}}}_{L^2_TL^6}^{\frac{2}{\gamma}(1-\xi)}\\
				&\leq C\norm{\sqrt{\varrho}}_{L^\infty_TH^1}^{2\xi}\norm{\varrho^{\frac{\gamma}{2}}}_{L^2_TH^1}^{\frac{2}{\gamma}(1-\xi)}\\
				&\leq C.
			\end{split}
		\end{align}
		Therefore, it follows from \eqref{u1.7} and \eqref{u1} that
		\begin{align*}
			\begin{split}
				\frac{d}{dt}\norm{\varrho^{\frac{1}{q+2}}v}_{L^{q+2}}^2&\leq C\left(\int \varrho^{q\gamma+2\gamma-q-1}dx\right)^{\frac{2}{q+2}}\\
				&=C\left(\int \varrho^{\zeta}\varrho^{\frac{q\gamma+2\gamma-q-6}{2}}dx\right)^{\frac{2}{q+2}}\\
				&\leq C\left(\int(\varrho^{\zeta}+1)dx\right)^{\frac{2}{q+2}}\\
				&\leq C\left(\int \varrho^\zeta dx\right)^{\frac{2}{q+2}}+C.
			\end{split}
		\end{align*}
		Integrating the above inequality over time and using \eqref{u2}, we obtain \eqref{u1.5}. This completes the proof of Proposition \ref{Prop high v}.
	\end{proof}
	
	Next we use the De Giorgi iteration to obtain an upper bound on the density. This approach originates from De Giorgi's method for improving the regularity of solutions to elliptic equations, which was later extended to parabolic equations by Ladyzhenskaya et al. in \cite{Ladyzhenskaia}, and has subsequently been elegantly applied by Haspot in \cite{Haspot-2017}. More recently, Yu and Wu \cite{Yu-Wu-2021} used it to establish an upper bound for the density in the two-dimensional Cauchy problem. Here, we adapt and extend this technique to establish the corresponding upper bound estimate in the three-dimensional case.
    
    For simplicity, we introduce the following notations:
	\begin{align*}
	&\varrho^{(k)}:=\max\{\varrho-k,0\},\text{ for }k\ge 2,\\
	&|\varrho^{(k)}|_{Q_t}:=\norm{\varrho^{(k)}}_{L^\infty_t L^2}+\norm{\nabla \varrho^{(k)}}_{L^2_tL^2},\\
	&\Gamma_k(t):=\{x\in\mathbb{T}^3\left|\varrho(x,t)\ge k \right.\},\\
	&\mu(k):=\int_0^T\mathcal{L}(\Gamma_k(t))^{\frac{q}{q+2}}dt,
	\end{align*}
    where \(\Gamma_k(t)\) denotes the upper level set of the density and \(\mathcal{L}(\Gamma_k(t))\) stands for the Lebesgue measure of \(\Gamma_k(t)\).
	Now we are ready to derive the space-time estimate of $\varrho^{(k)}$.
	\begin{lema}\label{Prop rho_k Q_T}
		There exists a constant \(C>0\) depending on \(T^*, q,\gamma, E_0\), and \(\|\varrho_0^{1/(q+2)}v_0\|_{L^{q+2}}\) such that for any $k\ge 2$,
			\begin{align}\label{rho_k Q_T}
			|\varrho^{(k)}|_{Q_T}\leq Ck\mu^{\frac{1}{2}}(k).
			\end{align}
	\end{lema}
	\begin{proof}
		For any $k\ge 2$, equation \(\eqref{Equ2}_1\) can be rewritten as
			\begin{align*}
			\partial_t(\varrho-k)+\divg(\varrho v)-\Delta(\varrho-k)=0.
		\end{align*}
		Multiplying the above equation by \(\varrho^{(k)}\), integrating the resulting equation over \(\mathbb{T}^3\), and using integration by parts and Young's inequality, we obtain
			\begin{align*}
			&\quad\frac{1}{2}\frac{d}{dt}\int |\varrho^{(k)}|^2dx+\int |\nabla\varrho^{(k)}|^2dx\\
			&=\int \nabla \varrho^{(k)}\cdot \varrho vdx\\
			&\leq \frac{1}{2}\int |\nabla\varrho^{(k)}|^2dx+\frac{1}{2}\int_{\Gamma_k(t)}\varrho^2|v|^2dx.
			\end{align*}
		Integrating this inequality in time yields
	 	\begin{align}\label{v norm'}
		\begin{split}
			|\varrho^{(k)}|_{Q_T}^2\leq C\int_0^T\int_{\Gamma_k(t)}\varrho^2|v|^2dxdt.
		\end{split}
		\end{align}
		Next, we estimate the right-hand side of the above inequality.  From \eqref{u1.5}, the definition of \(\mu(k)\), H\"{o}lder's inequality, and Young's inequality, we obtain
			\begin{align}\label{cc1'}
			\begin{split}
				&\quad C\int_0^T\int_{\Gamma_k(t)}\varrho^2|v|^2dxdt\\
				&\leq C\int_0^T\norm{\varrho v^{q+2}}_{L^1}^{\frac{2}{q+2}}\norm{\varrho}_{L^{\frac{2q+2}{q}}(\Gamma_k(t))}^{\frac{2q+2}{q+2}}dt\\
				&\leq C\int_0^T\Big(\norm{\varrho^{(k)}}_{L^{\frac{2q+2}{q}}(\Gamma_k(t))}^{\frac{2q+2}{q+2}}+\norm{k}_{L^{\frac{2q+2}{q}}(\Gamma_k(t))}^{\frac{2q+2}{q+2}}\Big)dt\\
				&\leq C\norm{\varrho^{(k)}}_{L_T^{\frac{2q+2}{q+2}}L^{\frac{2q+2}{q}}(\Gamma_k(t))}^{\frac{2q+2}{q+2}}+Ck^{\frac{2q+2}{q+2}}\mu(k)\\
				&\leq C\norm{\varrho^{(k)}}_{L^\infty_TL^2(\Gamma_k(t))}^{\frac{2q+2}{q+2}\frac{2q-2}{5q+4}}\norm{\varrho^{(k)}}_{L^2_TL^6(\Gamma_k(t))}^{\frac{2q+2}{q+2}\frac{3q+6}{5q+4}}\norm{\chi_{\Gamma_k(t)}}_{L^1_TL^{\frac{q+2}{q}}}^{\frac{2q+1}{5q+4}}+Ck^{\frac{2q+2}{q+2}}\mu(k)\\
				&\leq C|\varrho^{(k)}|_{Q_T}^{\frac{2q+2}{q+2}}\mu(k)^{\frac{2q+1}{5q+4}}+Ck^{\frac{2q+2}{q+2}}\mu(k)\\
				&\leq \frac{1}{2}|\varrho^{(k)}|_{Q_T}^{2}+C\mu(k)^{\frac{2q^2+5q+2}{5q+4}}+Ck^{\frac{2q+2}{q+2}}\mu(k)\\
				&\leq \frac{1}{2}|\varrho^{(k)}|_{Q_T}^{2}+Ck^2\mu(k),
			\end{split}
		\end{align}
         where \(\chi_{\Gamma_k(t)}\) is the indicator function of \(\Gamma_k(t)\).  Combining \eqref{cc1'} with \eqref{v norm'} yields \eqref{rho_k Q_T}, which completes the proof.
	\end{proof}
	
	We recall the following lemma from Ladyzhenskaya et al. \cite{Ladyzhenskaia} (Lemma 5.6, p.95); see also \cite{Haspot-2017} (Lemma 5.2) or \cite{Yu-Wu-2021} (Lemma 2.2).
	 \begin{lema}\label{Lem De}
		Assume that \(\Phi_n\), \(n=0,1,2,\dots\), is a non-negative sequence satisfying
		\begin{align*}
			\Phi_{n+1}\leq cb^n\Phi_n^{1+\eta},
		\end{align*}
		for some positive constants $c,\eta$, and $b>1$. If, in addition,
		\begin{align*}
			\Phi_0\leq \theta=c^{-1/\eta}b^{-1/\eta^2}, 
		\end{align*}
		then we have
		\begin{align*}
			\Phi_n\leq \theta b^{-n/\eta} \text{ and }\lim\limits_{n\to\infty}\Phi_n=0.
		\end{align*}
	\end{lema}
	With the above lemmas at hand, we can now establish an upper bound for the density.
	\begin{prop}\label{Prop R_T}
		There exists a constant \(C>0\) depending on \(T^*, q,\gamma, E_0, \|\varrho_0^{1/(q+2)}v_0\|_{L^{q+2}},\) and \(\norm{\varrho_0}_{L^\infty}\) such that 
		\begin{align}\label{3d RT}
			\sup_{0\leq t\leq T}\norm{\varrho}_{L^\infty}\leq C.
		\end{align}
	\end{prop}
	\begin{proof}
		We denote that
		\begin{align*}
			\hat{k}_0=\max\{2,\norm{\varrho_0}_{L^\infty}\},\quad k_n=M_{T^*}\hat{k}_0(2-2^{-n}),
		\end{align*}
		where \(n\in\mathbb N\) and \(M_{T^*}>1\) is a sufficiently large constant to be determined, depending only on \(T^*, q,\gamma, E_0, \|\varrho_0^{1/(q+2)}v_0\|_{L^{q+2}},\) and \( \|\varrho_0\|_{L^\infty}\). 
		A direct calculation shows that
			\begin{align}\label{u3}
			\norm{\varrho^{(k_n)}}_{L^{\frac{10q+8}{3q+6}}_TL^{\frac{10q+8}{3q}}}\leq C\norm{\varrho^{(k_n)}}_{L^\infty_T L^2}^{\frac{2q-2}{5q+4}}\norm{\varrho^{(k_n)}}_{L^2_TL^6}^{\frac{3q+6}{5q+4}}\leq C|\varrho^{(k_n)}|_{Q_T}.
			\end{align}
		Using Chebyshev's inequality, \eqref{u3}, together with \eqref{rho_k Q_T}, one has
			\begin{align}\label{u3.5}
			\begin{split}
				\mu(k_{n+1})&=\int_0^T\mathcal{L}(\varrho-k_n\ge k_{n+1}-k_n)^{\frac{q}{q+2}}dt\\
				&\leq \frac{1}{(k_{n+1}-k_n)^{\frac{10q+8}{3q+6}}}\int_0^T\left(\int |\varrho^{(k_n)}|^{\frac{10q+8}{3q}}dx\right)^{\frac{q}{q+2}}dt\\
				&=  \frac{1}{(k_{n+1}-k_n)^{\frac{10q+8}{3q+6}}}\norm{\varrho^{(k_n)}}_{L^{\frac{10q+8}{3q+6}}_TL^{\frac{10q+8}{3q}}}^{\frac{10q+8}{3q+6}}\\
				&\leq \frac{C}{(k_{n+1}-k_n)^{\frac{10q+8}{3q+6}}}|\varrho^{(k_n)}|_{Q_T}^{\frac{10q+8}{3q+6}}\\
				&\leq \frac{Ck_n^{\frac{10q+8}{3q+6}}}{(k_{n+1}-k_n)^{\frac{10q+8}{3q+6}}}\mu(k_n)^{\frac{5q+4}{3q+6}}\\
				&\leq C2^{\frac{10q+8}{3q+6}n}\mu(k_n)^{\frac{5q+4}{3q+6}}\\
				&\leq C16^n\mu(k_n)^{\frac{5q+4}{3q+6}}.
			\end{split}
			\end{align}
		Denoting by \(C_0\) the constant \(C\) appearing in the last inequality of \eqref{u3.5}, we obtain the relation
			\begin{align}\label{u4}
			\mu(k_{n+1})\leq C_016^n\mu(k_n)^{1+\frac{2q-2}{3q+6}}, \text{ for all }n\in\mathbb{N}.
			\end{align}
			Employing $\eqref{ene est}$ yields 
        \begin{align}\label{u4.4}
            \begin{split}
                \mu(k_0) &= \mu(M_{T^*}\hat{k}_0) 
                = \int_0^T \mathcal{L}\bigl(\Gamma_{M_{T^*}\hat{k}_0}(t)\bigr)^{\frac{q}{q+2}} \, dt \\
                &\le \int_0^T \left( \int \frac{\varrho}{M_{T^*}\hat{k}_0} \, dx \right)^{\frac{q}{q+2}} dt \\
                &\le C M_{T^*}^{-\frac{q}{q+2}} \left( \int_0^T \int \varrho \, dx dt \right)^{\frac{q}{q+2}} \\
                &\le C M_{T^*}^{-\frac{q}{q+2}}.
            \end{split}
        \end{align}
        Denoting the constant \(C\) appearing in the last line of \eqref{u4.4} by \(C_1\), we obtain
        \begin{align}\label{u4.5}
            \mu(k_0) \le C_1 M_{T^*}^{-\frac{q}{q+2}}.
        \end{align}
			To apply Lemma \ref{Lem De} to the sequence $\{\mu(k_n)\}_{n\in\mathbb{N}}$, we choose \(M_{T^*}\) sufficiently large so that
			\begin{align}\label{u4.6}
				C_1M_{T^*}^{-\frac{q}{q+2}}\leq C_0^{-\frac{3q+6}{2q-2}}16^{-(\frac{3q+6}{2q-2})^2}.
			\end{align}
		Therefore, combining \eqref{u4}, \eqref{u4.5}, and \eqref{u4.6}, an application of Lemma \ref{Lem De} yields
			\begin{align*}
			\lim\limits_{n\to\infty}\mu(k_n)=0.
			\end{align*}
		This limit forces $\mu(2M_{T^*}\hat{k}_0)=0$, which in turn implies that
        \begin{align*}
			\sup_{0\leq t\leq T}\norm{\varrho}_{L^\infty}\leq 2M_{T^*}\hat{k}_0.
		\end{align*}
		 This completes the proof of Proposition \ref{Prop R_T}.
	\end{proof}
	
	\section{Positive lower bound for the density (3D)}
	In this section we prove that the density has a positive lower bound on \(\mathbb{T}^3\times(0,T^*)\), which is the most important and challenging part of ensuring that the local strong solution exists globally in time.  
	
	We begin with Lemma~\ref{Prop v L^p}, which describes how the \(L^p\)-integrability of the effective velocity grows with respect to \(p\).
	\begin{lema}\label{Prop v L^p}
		Under the assumption that $p>2$, there exists a constant \(\hat{C}\ge 1\) depending on \(T^*,q,\gamma,E_0,\norm{\varrho_0^{1/(q+2)}v_0}_{L^{q+2}},\norm{\varrho_0}_{L^\infty}\), and \(\norm{v_0}_{L^\infty}\) but independent of \(p\), such that
		\begin{align}\label{u9}
			\sup_{0\leq t\leq T}\norm{\varrho^{\frac{1}{p+2}}v}_{L^{p+2}}\leq \hat{C}\sqrt{p+2}.
		\end{align}
	\end{lema}
	\begin{rmk}
		Indeed, the estimate of \(\norm{\varrho^{\frac{1}{p+2}}v}_{L^\infty_TL^{p+2}}\) depends on \(p\), and as \(p\to\infty\) this \(p\)-dependent bound tends to infinity.  This dependence prevents us from using the method of Haspot \cite{Haspot-2017} to obtain a positive lower bound on the density.
	\end{rmk}
	\begin{proof}[Proof of Lemma \ref{Prop v L^p}]
		For any $p>2$, multiplying equation \(\eqref{Equ2}_2\) by \(|v|^{p}v\) and integrating the resulting equation over \(\mathbb{T}^3\), we use integration by parts and Young's inequality to obtain
			\begin{align}\label{u5}
			\begin{split}
				&\quad\frac{1}{p+2}\frac{d}{dt}\int \varrho |v|^{p+2}dx+\int \varrho |v|^p|\nabla v|^2dx+p\int \varrho |v|^p|\nabla |v||^2dx\\
				&=\int \varrho^\gamma\divg(|v|^pv)dx\\
				&=\int \varrho^\gamma|v|^p\div vdx+ p\int \varrho^\gamma|v|^{p-1}v\cdot\nabla |v|dx\\
				&\leq \frac{1}{2}\int \varrho |v|^p|\nabla v|^2dx+\frac{p}{2}\int \varrho |v|^p|\nabla |v||^2dx+C(p+1)\int \varrho^{2\gamma-1}|v|^pdx.
			\end{split}
		\end{align}
		Therefore, by \eqref{3d RT} and \eqref{ent est}, we obtain
		\begin{align*}
			\begin{split}
				\frac{1}{p+2}\frac{d}{dt}\int \varrho |v|^{p+2}dx&\leq C(p+1)\int \varrho^{2\gamma-1}|v|^pdx\\
				&\leq C(p+1)\int \varrho |v|^p dx\\
				&\leq C(p+1)\left(\int \varrho |v|^{2}dx\right)^{\frac{2}{p}}\left(\int \varrho |v|^{p+2}dx\right)^{\frac{p-2}{p}}\\
				&\leq C(p+1)\norm{\varrho^{\frac{1}{p+2}}v}_{L^{p+2}}^{\frac{(p-2)(p+2)}{p}},
			\end{split}
		\end{align*}
		which shows that
		\begin{align*}
			\frac{d}{dt}\norm{\varrho^{\frac{1}{p+2}}v}_{L^{p+2}}^{\frac{2(p+2)}{p}}\leq C(p+1).
		\end{align*}
		Integrating the above inequality with respect to time, we obtain
			\begin{align*}
			\sup_{0\leq t\leq T}\norm{\varrho^{\frac{1}{p+2}}v}_{L^{p+2}}&\leq \left(\norm{\varrho_0^{\frac{1}{p+2}}v_0}_{L^{p+2}}^{\frac{2(p+2)}{p}}+C(p+1)\right)^{\frac{p}{2(p+2)}}\\
			&\leq \norm{\varrho_0^{\frac{1}{p+2}}v_0}_{L^{p+2}}+C(p+1)^{\frac{p}{2(p+2)}}\\
			&\leq \norm{\varrho_0^{\frac{1}{2}}v_0}_{L^2}^{\frac{2}{p+2}}\norm{v_0}_{L^\infty}^{\frac{p}{p+2}}+C\sqrt{p+1}\\
			&\leq C\sqrt{p+2}.
			\end{align*}
		
		This completes the proof of Lemma \ref{Prop v L^p}.
	\end{proof}
	
	The following proposition is key to obtaining the positive lower bound on the density. It establishes a control relation between the effective velocity and the lower bound of the density.  We introduce the following notation:
	\begin{align}\label{u10}
		\tau:=\varrho^{-1},\quad \Theta_T:=\sup_{0\leq t\leq T}\norm{\tau}_{L^\infty}+e^{\frac{25}{9}}.
	\end{align}
	\begin{prop}\label{Prop v inf}
		There exists a constant \(C_v\ge 1\) depending on \(T^*,q,\gamma,E_0,\norm{\varrho_0^{1/(q+2)}v_0}_{L^{q+2}}\), \(\norm{\varrho_0}_{L^\infty}\), and \(\norm{v_0}_{L^\infty}\) such that
		\begin{align}\label{Cv}
			\sup_{0\leq t\leq T}\norm{v}_{L^\infty}\leq C_v\sqrt{\log \Theta_T}.
		\end{align}
	\end{prop}
	\begin{proof}
		For any \(p>0\), multiplying equation \(\eqref{Equ2}_2\) by \(|v|^{p}v\), integrating the resulting equation over \(\mathbb{T}^3\), and using integration by parts, we obtain from \eqref{u5}
		\begin{align*}
			\begin{split}
				&\quad\frac{1}{p+2}\frac{d}{dt}\int \varrho |v|^{p+2}dx+\frac{1}{2}\int \varrho |v|^p|\nabla v|^2dx+\frac{p}{2}\int \varrho |v|^p|\nabla |v||^2dx\\
				&\leq C(p+1)\int \varrho^{2\gamma-1}|v|^pdx\\
				&\leq C (p+1)\int \varrho |v|^p dx\\
				&\leq C (p+1)\left(\int \varrho |v|^2dx\right)^{\frac{2}{p}}\left(\int \varrho |v|^{p+2}dx\right)^{\frac{p-2}{p}}\\
				&\leq C (p+1)\left(\int \varrho |v|^{p+2}dx\right)^{\frac{p-2}{p}},
			\end{split}
		\end{align*}
		where we have used \eqref{3d RT} and \eqref{ent est}. Therefore, integrating the above inequality with respect to time and using H\"{o}lder's inequality, we get
			\begin{align}\label{u6}
			\begin{split}
				&\quad\sup_{0\leq t\leq T}\frac{1}{p+2}\int \varrho |v|^{p+2}dx+\frac{1}{2}\int_{Q_T} \varrho |v|^p|\nabla v|^2dxdt+\frac{p}{2}\int_{Q_T} \varrho |v|^p|\nabla |v||^2dxdt\\
				&\leq C(p+1)\int_0^T\left(\int \varrho |v|^{p+2}dx\right)^{\frac{p-2}{p}}dt+\frac{2}{p+2}\int \varrho_0 |v_0|^{p+2}dx\\
				&\leq C(p+1)\left(\int_{Q_T} \varrho |v|^{p+2}dxdt\right)^{\frac{p-2}{p}}+\frac{2}{p+2}\int \varrho_0 |v_0|^{p+2}dx.
			\end{split}
			\end{align}

			Next, we establish a reverse H\"{o}lder's inequality.  From \eqref{3d RT}, \eqref{u6}, H\"{o}lder's inequality and the Sobolev embedding, we obtain
			\begin{align}\label{u7}
				\begin{split}
					&\quad\int_{Q_T} \varrho|v|^{\frac{5}{3}(p+2)}dxdt\\
					&\leq\int_0^T\left(\int \varrho^{\frac{3}{2}}|v|^{p+2}dx\right)^{\frac{2}{3}}\left(\int |v|^{3(p+2)}dx\right)^{\frac{1}{3}}dt \\
					&\leq C\sup_{0\leq t\leq T}\left(\int \varrho|v|^{p+2} dx\right)^{\frac{2}{3}}\left(\int_{Q_T} |\nabla |v|^{\frac{p+2}{2}}|^2dxdt+\int_{Q_T}|v|^{p+2}dxdt\right)\\
					&\leq C\Theta_T\left(\sup_{0\leq t\leq T}\int \varrho |v|^{p+2}dx+\int_{Q_T}\varrho|\nabla|v|^{\frac{p+2}{2}}|^2dxdt\right)^{\frac{5}{3}}\\
					&\leq C\Theta_T\left(C(p+2)^2\left(\int_{Q_T} \varrho |v|^{p+2}dxdt\right)^{\frac{p-2}{p}}+C\int \varrho_0 |v_0|^{p+2}dx\right)^{\frac{5}{3}}\\
					&\leq C\Theta_T(p+2)^{\frac{10}{3}}\left(\int_{Q_T}\varrho|v|^{p+2}dxdt\right)^{\frac{5}{3}}+C\Theta_T(p+2)^{\frac{10}{3}}+C\Theta_T\left(\int \varrho_0 |v_0|^{p+2}dx\right)^{\frac{5}{3}}.
				\end{split}
			\end{align}
			A direct calculation shows that
			\begin{align}\label{u7.1}
				\norm{\varrho_0^{\frac{1}{p+2}}v_0}_{L^{p+2}}\leq \norm{\varrho_0^{\frac{1}{2}}v_0}_{L^2}^{\frac{2}{p+2}}\norm{v_0}_{L^\infty}^{\frac{p}{p+2}}\leq \norm{\varrho_0^{\frac{1}{2}}v_0}_{L^2}+\norm{v_0}_{L^\infty}.
			\end{align}
			Denote the constant \(C\) in the last inequality of \eqref{u7} by \(C_2\), and assume without loss of generality that \(C_2\ge 1\).  For simplicity, we introduce the following notation. For any integer \(k \in \mathbb{N}\) with \(k > l\), where \(l\ge 3\) is a positive integer to be determined later, we define
			\begin{align*}
				\begin{split}
					&r=\frac{5}{3},\quad p_k+2=r^k,\quad C_3=\norm{\varrho_0^{\frac{1}{2}}v_0}_{L^2}+\norm{v_0}_{L^\infty}+1,\\
					&\Xi(k)=\int_{Q_T} \varrho|v|^{r^k}dxdt.
				\end{split}
			\end{align*}
			Substituting \(p\) with \(p_k\) in \(\eqref{u7}\), we use \eqref{u7.1}  to get
			\begin{align*}
				\begin{split}
					\Xi(k+1)\leq C_2\Theta_T(r^{\frac{10}{3}k}\Xi(k)^{r}+r^{\frac{10}{3}k}+C_3^{r^{k+1}}).
				\end{split}
			\end{align*}
			Let \(\tilde{\Xi}(k) = \max\{\Xi(k), C_3^{r^k}\}\). Then a direct calculation shows that
			\begin{align}\label{u7.9}
				\begin{split}
					\tilde{\Xi}(k+1)&=\max\{\Xi(k+1),C_3^{r^{k+1}} \}\\
					&\leq \max\{C_2\Theta_Tr^{\frac{10}{3}k}\Xi(k)^r+C_2\Theta_Tr^{\frac{10}{3}k}+C_2\Theta_TC_3^{r^{k+1}},C_3^{r^{k+1}}\}\\
					&\leq \max\{C_2\Theta_Tr^{\frac{10}{3}k}\tilde{\Xi}(k)^r+C_2\Theta_Tr^{\frac{10}{3}k}\tilde{\Xi}(k)^r+C_2\Theta_Tr^{\frac{10}{3}k}\tilde{\Xi}(k)^r,\tilde{\Xi}(k)^r\}\\
					&\leq 3C_2\Theta_Tr^{\frac{10}{3}k}\tilde{\Xi}(k)^r.
				\end{split}
			\end{align}
            Hence, it holds that
            \begin{align}\label{u8}
			\begin{split}
				\tilde{\Xi}(k+1)^{\frac{1}{r^{k+1}}}&\leq (3C_2\Theta_T)^{\frac{1}{r^{k+1}}}r^{\frac{10}{3}\frac{k}{r^{k+1}}}\tilde{\Xi}(k)^{\frac{1}{r^k}}\\
				&\leq (3C_2\Theta_T)^{\frac{1}{r^{k+1}}+\frac{1}{r^k}}r^{\frac{10}{3}(\frac{k}{r^{k+1}}+\frac{k-1}{r^{k}})}\tilde{\Xi}(k-1)^{\frac{1}{r^{k-1}}}\\
				&\leq \dots\\
				&\leq (3C_2\Theta_T)^{\sum_{i=l+1}^{k+1}r^{-i}}r^{\frac{10}{3}\sum_{i=l+1}^{k+1}(i-1)r^{-i}}\tilde{\Xi}(l)^{\frac{1}{r^l}}.
			\end{split}
		\end{align}

			Finally, we will choose an appropriate \(l\) as the starting point of the iteration and proceed by iteration using the reverse H\"{o}lder's inequality $\eqref{u8}$. We choose \(l\) as follows:
		\begin{align*}
			l = \bigl\lfloor \log_{r}(\log \Theta_T)+1 \bigr\rfloor,
		\end{align*}
		where \(\lfloor \cdot \rfloor\) denotes the floor function. From \(\eqref{u8}\) and \(\eqref{u9}\), we obtain
		\begin{align*}
			\begin{split}
				\tilde{\Xi}(k+1)^{\frac{1}{r^{k+1}}}&\leq(3C_2)^{\sum_{i=1}^{\infty}r^{-i}}r^{\frac{10}{3}\sum_{i=1}^{\infty}(i-1)r^{-i}}\Theta_T^{\sum_{i=l+1}^{\infty}r^{-i}}\tilde{\Xi}(l)^{\frac{1}{r^l}}\\
				&= (3C_2)^{\sum_{i=1}^{\infty}r^{-i}}r^{\frac{10}{3}\sum_{i=1}^{\infty}(i-1)r^{-i}}\Theta_T^{\frac{1}{r-1}r^{-l}}\max\{\Xi(l)^{\frac{1}{r^l}},C_3\}\\
				&\leq (3C_2)^{\sum_{i=1}^{\infty}r^{-i}}r^{\frac{10}{3}\sum_{i=1}^{\infty}(i-1)r^{-i}}\Theta_T^{\frac{1}{r-1}r^{-{\log_{r}(\log \Theta_T)}}}\max\{T^{\frac{1}{r^l}}\hat{C}r^{\frac{l}{2}},C_3\}\\
				&\leq (1+T^*)\hat{C}C_3(3C_2)^{\sum_{i=1}^{\infty}r^{-i}}r^{\frac{10}{3}\sum_{i=1}^{\infty}(i-1)r^{-i}}\Theta_T^{\frac{1}{r-1}r^{-{\log_{r}(\log \Theta_T)}}}r^{\frac{l}{2}}\\
				&\leq  (1+T^*)\hat{C}C_3(3C_2)^{\sum_{i=1}^{\infty}r^{-i}}r^{\frac{10}{3}\sum_{i=1}^{\infty}(i-1)r^{-i}}\Theta_T^{\frac{3}{2}\frac{1}{\log \Theta_T}}r^{\frac{\log_{r}(\log \Theta_T)+1}{2}}\\
				&\leq r^{\frac{1}{2}}(1+T^*)\hat{C}C_3(3C_2)^{\sum_{i=1}^{\infty}r^{-i}}r^{\frac{10}{3}\sum_{i=1}^{\infty}(i-1)r^{-i}}\Theta_T^{\frac{3}{2}\frac{1}{\log \Theta_T}}(\log \Theta_T)^{\frac{1}{2}}\\
				&\leq r^{\frac{1}{2}}(1+T^*)\hat{C}C_3(3C_2)^{\sum_{i=1}^{\infty}r^{-i}}r^{\frac{10}{3}\sum_{i=1}^{\infty}(i-1)r^{-i}}e^{\frac{3}{2}}\sqrt{\log \Theta_T}.
			\end{split}
		\end{align*}
		Combined with the definition of \(\tilde{\Xi}(\cdot)\), it shows that for any \(k \in \mathbb{N}^{+}\) with \(k > l\),
		\begin{align*}
			\left(\int_{Q_T} \varrho|v|^{r^{k+1}}dxdt\right)^{\frac{1}{r^{k+1}}}\leq r^{\frac{1}{2}}(1+T^*)\hat{C}C_3(3C_2)^{\sum_{i=1}^{\infty}r^{-i}}r^{\frac{10}{3}\sum_{i=1}^{\infty}(i-1)r^{-i}}e^{\frac{3}{2}}\sqrt{\log \Theta_T},
		\end{align*}
		which implies that
		\begin{align*}
			\left(\int_{Q_T} |v|^{r^{k+1}}dxdt\right)^{\frac{1}{r^{k+1}}}\leq \Theta_T^{\frac{1}{r^{k+1}}} r^{\frac{1}{2}}(1+T^*)\hat{C}C_3(3C_2)^{\sum_{i=1}^{\infty}r^{-i}}r^{\frac{10}{3}\sum_{i=1}^{\infty}(i-1)r^{-i}}e^{\frac{3}{2}}\sqrt{\log \Theta_T}.
		\end{align*}
		Letting \(k \to \infty\), we obtain
		\begin{align*}
			\sup_{0\leq t\leq T}\norm{v}_{L^\infty}\leq r^{\frac{1}{2}}(1+T^*)\hat{C}C_3(3C_2)^{\sum_{i=1}^{\infty}r^{-i}}r^{\frac{10}{3}\sum_{i=1}^{\infty}(i-1)r^{-i}}e^{\frac{3}{2}}\sqrt{\log \Theta_T}.
		\end{align*}
		
		Therefore, the proof of Proposition \ref{Prop v inf} is complete.
	\end{proof}

    We are now in a position to prove that the density has a positive lower bound.
	\begin{prop}\label{Prop A_T}
		There exists a constant \(C>0\) depending on \(T^*,q,\gamma,E_0,\norm{\varrho_0^{1/(q+2)}v_0}_{L^{q+2}}\), \(\norm{\varrho_0}_{L^\infty}, \norm{v_0}_{L^\infty}\), and \(\norm{\tau_0}_{L^\infty}\) such that
		\begin{align}\label{3d A_T}
			\sup_{0\leq t\leq T}\norm{\tau}_{L^\infty}\leq C.
		\end{align}
	\end{prop}
	\begin{proof}
		\textbf{Step 1.} We claim that there exists a \(T_0 > 0\) sufficiently small, depending on \(T^*, q,\gamma, E_0, \|\varrho_0^{1/(q+2)}v_0\|_{L^{q+2}}, \|\varrho_0\|_{L^\infty},\) and \( \|v_0\|_{L^\infty}\); and a constant \(C > 0\) sufficiently large, depending on \(T^*, q,\gamma, E_0, \|\varrho_0^{1/(q+2)}v_0\|_{L^{q+2}}, \|\varrho_0\|_{L^\infty}, \|v_0\|_{L^\infty}\), and \(\|\tau_0\|_{L^\infty}\), such that
		\begin{align}
			\Theta_{T_0}\leq C.
		\end{align}
		
		From \(\eqref{Equ2}_1\), we know that \(\tau\) satisfies the following equation:
		\begin{align}\label{Equ A}
			\partial_t \tau-\Delta \tau +2\tau^{-1}|\nabla \tau|^2+v\cdot \nabla \tau-\tau\divg v=0.
		\end{align}
		For any \(p>0\), multiplying equation \(\eqref{Equ A}\) by \(\tau^{p+1}\), integrating the resulting equation over \(\mathbb{T}^3\), and using integration by parts and Young's inequality, we obtain that
			\begin{align}\label{u13}
			\begin{split}
				&\quad\frac{1}{p+2}\frac{d}{dt}\int \tau^{p+2}dx+(p+3)\int \tau^p|\nabla \tau|^2dx\\
				&\leq (p+3)\int |v||\nabla \tau|\tau^{p+1}dx\\
				&\leq \frac{p+3}{2}\int \tau^p|\nabla \tau|^2dx+(p+2)\int |v|^2\tau^{p+2}dx\\
				&\leq \frac{p+3}{2}\int \tau^p|\nabla \tau|^2dx+(p+2)\norm{v}_{L^\infty(Q_T)}^2\int \tau^{p+2}dx.
			\end{split}
		\end{align}
		Integrating the above inequality with respect to time yields
			\begin{align}\label{u11}
			\begin{split}
				&\quad\sup_{0\leq t\leq T}\frac{1}{p+2}\int \tau^{p+2}dx+\frac{p+3}{2}\int_{Q_T} \tau^p|\nabla \tau|^2dxdt\\
				&\leq 2(p+2)\norm{v}_{L^\infty(Q_T)}^2\int_{Q_T} \tau^{p+2}dxdt+\frac{2}{p+2}\int\tau_0^{p+2}dx.
			\end{split}
		\end{align}
		Next, we establish a reverse H\"{o}lder's inequality. Using \(\eqref{u11}\) and the Sobolev embedding, one has
			\begin{align}\label{u12}
			\begin{split}
				&\quad\int_{Q_T} \tau^{\frac{5}{3}(p+2)}dxdt\\
				&\leq\int_0^T\left(\int \tau^{p+2}dx\right)^{\frac{2}{3}}\left(\int \tau^{3(p+2)}dx\right)^{\frac{1}{3}}dt \\
				&\leq C\sup_{0\leq t\leq T}\left(\int \tau^{p+2} dx\right)^{\frac{2}{3}}\left(\int_{Q_T} |\nabla \tau^{\frac{p+2}{2}}|^2dxdt+\int_{Q_T}\tau^{p+2}dxdt\right)\\
				&\leq C\left(\sup_{0\leq t\leq T}\int \tau^{p+2}dx+\int_{Q_T} |\nabla\tau^{\frac{p+2}{2}}|^2dxdt\right)^{\frac{5}{3}}\\
				&\leq C\left((p+2)^2\norm{v}_{L^\infty(Q_T)}^2\int_{Q_T}\tau^{p+2}dxdt+\norm{\tau_0}_{L^{p+2}}^{p+2}\right)^{\frac{5}{3}}\\
				&\leq C(p+2)^{\frac{10}{3}}\norm{v}_{L^\infty(Q_T)}^{\frac{10}{3}}\left(\int_{Q_T}\tau^{p+2}dxdt\right)^{\frac{5}{3}}+C\norm{\tau_0}_{L^{p+2}}^{\frac{5}{3}(p+2)}\\
				&\leq C(p+2)^{\frac{10}{3}}\norm{v}_{L^\infty(Q_T)}^{\frac{10}{3}}\left(\int_{Q_T}\tau^{p+2}dxdt\right)^{\frac{5}{3}}+C(\norm{\tau_0}_{L^2}+\norm{\tau_0}_{L^\infty})^{\frac{5}{3}(p+2)}.
			\end{split}
			\end{align}
		We denote the constant \(C\) in the last inequality of $\eqref{u12}$ by \(C_4\). Without loss of generality, it may be assumed that \(C_4 \ge 1\). We introduce the following notation for any \(k \in \mathbb{N}^+\) with \(k \ge 2\):
		\begin{align*}
			\begin{split}
				&r = \frac{5}{3}, \quad p_k + 2 = r^k, \quad C_5 = \|\tau_0\|_{L^2} + \|\tau_0\|_{L^\infty} + 1, \\
				&\Upsilon(k) = \int_{Q_T} \tau^{r^k} \, dxdt.
			\end{split}
		\end{align*}
		Substituting \(p\) with \(p_k\) in \(\eqref{u12}\), we obtain
		\begin{align}\label{u12.1}
			\begin{split}
				\Upsilon(k+1)\leq C_4r^{\frac{10}{3}k}\norm{v}_{L^\infty(Q_T)}^{\frac{10}{3}}\Upsilon(k)^r+C_4C_5^{r^{k+1}}.
			\end{split}
		\end{align}
		Let \(\tilde{\Upsilon}(k) = \max\{\Upsilon(k), C_5^{r^k}\}\). Applying \(\eqref{Cv}\) and $\eqref{u12.1}$ yields that
		\begin{align*}
			\begin{split}
				\tilde{\Upsilon}(k+1)&=\max\{\Upsilon(k+1),C_5^{r^{k+1}}\}\\
				&\leq \max\{C_4r^{\frac{10}{3}k}\norm{v}_{L^\infty(Q_T)}^{\frac{10}{3}}\Upsilon(k)^r+C_4C_5^{r^{k+1}},C_5^{r^{k+1}}\}\\
				&\leq \max\{C_4r^{\frac{10}{3}k}\norm{v}_{L^\infty(Q_T)}^{\frac{10}{3}}\tilde{\Upsilon}(k)^r+C_4\tilde{\Upsilon}(k)^r,\tilde{\Upsilon}(k)^r\}\\
				&\leq C_4r^{\frac{10}{3}k}(\norm{v}_{L^\infty(Q_T)}^{\frac{10}{3}}+1)\tilde{\Upsilon}(k)^r\\
				&\leq 2C_4C_v^{\frac{10}{3}}r^{\frac{10}{3}k}(\log \Theta_T)^{\frac{5}{3}}\tilde{\Upsilon}(k)^r,
			\end{split}
		\end{align*}
		which implies that
		\begin{align}\label{u16}
			\begin{split}
				\tilde{\Upsilon}(k+1)^{\frac{1}{r^{k+1}}}&\leq (2C_4C_v^{\frac{10}{3}})^{\frac{1}{r^{k+1}}}r^{\frac{10}{3}\frac{k}{r^{k+1}}}(\log \Theta_T)^{\frac{5}{3}\frac{1}{r^{k+1}}}\tilde{\Upsilon}(k)^{\frac{1}{r^k}}\\
				&\leq (2C_4C_v^{\frac{10}{3}})^{\frac{1}{r^{k+1}}+\frac{1}{r^k}}r^{\frac{10}{3}(\frac{k}{r^{k+1}}+\frac{k-1}{r^k})}(\log \Theta_T)^{\frac{5}{3}(\frac{1}{r^{k+1}}+\frac{1}{r^k})}\tilde{\Upsilon}(k-1)^{\frac{1}{r^{k-1}}}\\
				&\leq \dots\\
				&\leq (2C_4C_v^{\frac{10}{3}})^{\sum_{i=3}^{k+1}r^{-i}}r^{\frac{10}{3}\sum_{i=3}^{k+1}r^{-i}(i-1)}(\log \Theta_T)^{\frac{5}{3}\sum_{i=3}^{k+1}r^{-i}}\tilde{\Upsilon}(2)^{\frac{1}{r^2}}\\
				&\leq (2C_4C_v^{\frac{10}{3}})^{\sum_{i=1}^{\infty}r^{-i}}r^{\frac{10}{3}\sum_{i=1}^{\infty}r^{-i}(i-1)}(\log \Theta_T)^{\frac{9}{10}}\tilde{\Upsilon}(2)^{\frac{1}{r^2}}.
			\end{split}
		\end{align}
		Next, we estimate \(\tilde{\Upsilon}(2)^{\frac{1}{r^2}}\). This requires estimating \(\Upsilon(2)^{\frac{1}{r^2}}\). Substituting \(p = \frac{7}{9}\) into \(\eqref{u13}\) gives
		\begin{align}\label{u19}
			\begin{split}
				\frac{d}{dt}\int \tau^{r^2}dx\leq r^4\norm{v}_{L^\infty(Q_T)}^2\int \tau^{r^2}dx,
			\end{split}
		\end{align}
		 which together with Gr\"{o}nwall's inequality yields
		 \begin{align}\label{u14}
		 	\begin{split}
		 		\sup_{0\leq t\leq T}\int \tau^{r^2}dx\leq \exp\{r^4T\norm{v}_{L^\infty(Q_T)}^2\}\int \tau_0^{r^2}dx.
		 	\end{split}
		 \end{align}
		 Using \(\eqref{u14}\) together with \(\eqref{Cv}\), we obtain
		 \begin{align}\label{u15}
		 	\begin{split}
		 		\Upsilon(2)^{\frac{1}{r^2}}&=\left(\int_{Q_T}\tau^{r^2}dxdt\right)^{\frac{1}{r^2}}\\
		 		&\leq T^{\frac{1}{r^2}}\left(\sup_{0\leq t\leq T}\int \tau^{r^2}dx\right)^{\frac{1}{r^2}}\\
		 		&\leq (T^*)^{\frac{1}{r^2}}\exp\{r^2T\norm{v}_{L^\infty(Q_T)}^2\}\left(\int \tau_0^{r^2}dx\right)^{\frac{1}{r^2}}\\
		 		&\leq (T^*)^{\frac{1}{r^2}}\exp\{r^2TC_v^2\log \Theta_T\}\norm{\tau_0}_{L^{r^2}}\\
		 		&= (T^*)^{\frac{1}{r^2}}\Theta_T^{r^2TC_v^2}\norm{\tau_0}_{L^{r^2}}.
		 	\end{split}
		 \end{align}
		 Substituting \eqref{u15} into \eqref{u16} yields
		 \begin{align}\label{u17}
		 	\begin{split}
                \tilde{\Upsilon}(k+1)^{\frac{1}{r^{k+1}}}&\leq (2C_4C_v^{\frac{10}{3}})^{\sum_{i=1}^{\infty}r^{-i}}r^{\frac{10}{3}\sum_{i=1}^{\infty}r^{-i}(i-1)}(\log \Theta_T)^{\frac{9}{10}}\tilde{\Upsilon}(2)^{\frac{1}{r^2}}\\
		 		&\leq (2C_4C_v^{\frac{10}{3}})^{\sum_{i=1}^{\infty}r^{-i}}r^{\frac{10}{3}\sum_{i=1}^{\infty}r^{-i}(i-1)}\Theta_T^{\frac{1}{2}}\max\{(T^*)^{\frac{1}{r^2}}\Theta_T^{r^2TC_v^2}\norm{\tau_0}_{L^{r^2}},C_5\}\\
		 		&\leq (1+(T^*)^{\frac{1}{r^2}})C_5(2C_4C_v^{\frac{10}{3}})^{\sum_{i=1}^{\infty}r^{-i}}r^{\frac{10}{3}\sum_{i=1}^{\infty}r^{-i}(i-1)}(\norm{\tau_0}_{L^{r^2}}+1)\Theta_T^{\frac{1}{2}+r^2TC_v^2}.
		 	\end{split}
		 \end{align}
		 Let \(C_6\) be a constant depending on \(T^*, q,\gamma, E_0, \|\varrho_0^{1/(q+2)}v_0\|_{L^{q+2}}, \|\varrho_0\|_{L^\infty}, \|v_0\|_{L^\infty}\), and \(\|\tau_0\|_{L^\infty}\), defined by
		 \begin{align*}
		 	C_6 :=\bigl(1+(T^*)^{\frac{1}{r^2}}\bigr) C_5 \bigl(2C_4C_v^{\frac{10}{3}}\bigr)^{\sum_{i=1}^{\infty} r^{-i}} 
		 	r^{\frac{10}{3} \sum_{i=1}^{\infty} r^{-i}(i-1)} \bigl(\|\tau_0\|_{L^{r^2}} + 1\bigr).
		 \end{align*}
		 Inequality \(\eqref{u17}\) shows that for any integer $k\ge 2$,
		 \begin{align*}
		 	\begin{split}
		 		\left(\int_{Q_T}\tau^{r^{k+1}}dxdt\right)^{\frac{1}{r^{k+1}}}\leq C_6 \Theta_T^{\frac{1}{2}+r^2TC_v^2}.
		 	\end{split}
		 \end{align*}
		 Letting \(k \to \infty\), we obtain
		 \begin{align*}
		 	\sup_{0\leq t\leq T}\norm{\tau}_{L^\infty}\leq C_6 \Theta_T^{\frac{1}{2}+r^2TC_v^2},
		 \end{align*}
		 which, together with the definition of \(\Theta_T\) in $\eqref{u10}$, implies that
		\begin{align}\label{u18}
		\Theta_T\leq (C_6+e^{\frac{25}{9}})\Theta_T^{\frac{1}{2}+r^2TC_v^2}.
		\end{align}
		We choose \(T_0 > 0\) sufficiently small, depending on \(T^*, q,\gamma, E_0, \|\varrho_0^{1/(q+2)}v_0\|_{L^{q+2}}, \|\varrho_0\|_{L^\infty}\), and \(\|v_0\|_{L^\infty}\) such that
		\begin{align}\label{u22}
			r^2 T_0 C_v^2 \le \frac{1}{4}.
		\end{align}
		Hence, from \(\eqref{u18}\) and $\eqref{u22}$ we deduce that
		\begin{align*}
			\Theta_{T_0} &\le (C_6 + e^{\frac{25}{9}}) \Theta_{T_0}^{\frac{3}{4}} \\
			&\le \frac{3}{4} \Theta_{T_0} + \frac{1}{4} (C_6 + e^{\frac{25}{9}})^4,
		\end{align*}
		which implies that
		\begin{align*}
			\Theta_{T_0} \le (C_6 + e^{\frac{25}{9}})^4.
		\end{align*}
		
		\textbf{Step 2.} Without loss of generality, we assume \(2T_0 < T^*\). We claim that there exists a constant \(C > 0\) sufficiently large, depending on \(T^*, q, \gamma, E_0, \|\varrho_0^{1/(q+2)}v_0\|_{L^{q+2}}, \|\varrho_0\|_{L^\infty}, \|v_0\|_{L^\infty}\), and \(\|\tau_0\|_{L^\infty}\), such that
		\begin{align*}
			\Theta_{2T_0}\leq C.
		\end{align*}
		
		For any \(T \in (T_0, 2T_0)\), we deduce from \(\eqref{u19}\) that
		\begin{align}\label{u20}
			\begin{split}
				\sup_{T_0\leq t\leq T}\int \tau^{r^2}dx&\leq \exp\{r^4(T-T_0)\norm{v}_{L^\infty(Q_T)}^2\}\int \tau(T_0)^{r^2}dx\\
				&\leq \exp\{r^4T_0\norm{v}_{L^\infty(Q_T)}^2\}\int \tau(T_0)^{r^2}dx.
			\end{split}
		\end{align}
		Combining \(\eqref{u20}\), \(\eqref{Cv}\), and \(\eqref{u22}\), we get
		\begin{align}\label{u21}
			\begin{split}
				\Upsilon(2)^{\frac{1}{r^2}}&=\left(\int_{Q_T}\tau^{r^2}dxdt\right)^{\frac{1}{r^2}}\\
				&\leq T^{\frac{1}{r^2}}\left(\sup_{0\leq t\leq T_0}\int \tau^{r^2}dx+\sup_{T_0\leq t\leq T}\int \tau^{r^2}dx\right)^{\frac{1}{r^2}}\\
				&\leq (T^*)^{\frac{1}{r^2}}\sup_{0\leq t\leq T_0}\left(\int \tau^{    r^2}dx\right)^{\frac{1}{r^2}}+(T^*)^{\frac{1}{r^2}}\sup_{T_0\leq t\leq T}\left(\int \tau^{r^2}dx\right)^{\frac{1}{r^2}}\\
				&\leq (T^*)^{\frac{1}{r^2}}\Theta_{T_0}+(T^*)^{\frac{1}{r^2}}\exp\{r^2T_0C_v^2\log \Theta_T\}\norm{\tau(T_0)}_{L^{r^2}}\\
				&\leq (T^*)^{\frac{1}{r^2}}\Theta_{T_0}+(T^*)^{\frac{1}{r^2}}\Theta_T^{\frac{1}{4}}\norm{\tau(T_0)}_{L^{r^2}}.
			\end{split}
		\end{align}	
		It follows from \(\eqref{u16}\) and \(\eqref{u21}\) that for any $k\ge 2$,
		\begin{align}\label{u23}
			\begin{split}
				&\quad\left(\int_{Q_T}\tau^{r^{k+1}}dxdt\right)^{\frac{1}{r^{k+1}}}\leq\tilde{\Upsilon}(k+1)^{\frac{1}{r^{k+1}}}\\
				&\leq (2C_4C_v^{\frac{10}{3}})^{\sum_{i=1}^{\infty}r^{-i}}r^{\frac{10}{3}\sum_{i=1}^{\infty}r^{-i}(i-1)}(\log \Theta_T)^{\frac{9}{10}}\tilde{\Upsilon}(2)^{\frac{1}{r^2}}\\
				&\leq (2C_4C_v^{\frac{10}{3}})^{\sum_{i=1}^{\infty}r^{-i}}r^{\frac{10}{3}\sum_{i=1}^{\infty}r^{-i}(i-1)}\Theta_T^{\frac{1}{2}}\max\{(T^*)^{\frac{1}{r^2}}\Theta_{T_0}+(T^*)^{\frac{1}{r^2}}\Theta_T^{\frac{1}{4}}\norm{\tau(T_0)}_{L^{r^2}},C_5\}\\
				&\leq 2C_5((1+(T^*)^{\frac{1}{r^2}})\Theta_{T_0}(1+\norm{\tau(T_0)}_{L^{r^2}}))(2C_4C_v^{\frac{10}{3}})^{\sum_{i=1}^{\infty}r^{-i}}r^{\frac{10}{3}\sum_{i=1}^{\infty}r^{-i}(i-1)}\Theta_T^{\frac{3}{4}}.
			\end{split}
		\end{align}
		Let \(C_7\) be a constant depending on \(T^*, q,\gamma, E_0, \|\varrho_0^{1/(q+2)}v_0\|_{L^{q+2}}, \|\varrho_0\|_{L^\infty}, \|v_0\|_{L^\infty}\), and \(\|\tau_0\|_{L^\infty}\), defined by
		\begin{align*}
			C_7 := 2C_5\Bigl((1+(T^*)^{\frac{1}{r^2}}) \Theta_{T_0} \bigl(1+\|\tau(T_0)\|_{L^{r^2}}\bigr)\Bigr) 
			\bigl(2C_4 C_v^{\frac{10}{3}}\bigr)^{\sum_{i=1}^{\infty} r^{-i}} 
			r^{\frac{10}{3} \sum_{i=1}^{\infty} r^{-i}(i-1)}.
		\end{align*}
		Letting \(k \to \infty\) in \(\eqref{u23}\), we obtain
		\begin{align*}
				\sup_{0\leq t\leq T}\norm{\tau}_{L^\infty}\leq C_7 \Theta_T^{\frac{3}{4}},
		\end{align*}
		which, combined with the definition of \(\Theta_T\), implies that
		\begin{align*}
			\Theta_T\leq (C_7+e^{\frac{25}{9}})\Theta_T^{\frac{3}{4}}.
		\end{align*}
		Hence, applying Young’s inequality and using the arbitrariness of \(T\), one has
		\begin{align*}
			\Theta_{2T_0}\leq (C_7+e^{\frac{25}{9}})^4.
		\end{align*}
		
		\textbf{Step 3.} There exists some \(j \in \mathbb{N}^+\) such that
		\begin{align*}
			jT_0 \le T^* < (j+1)T_0, 
		\end{align*}
		Repeating the argument in Step 2, we can prove that there exists a constant \(C > 0\) depending on \(T^*, q, \gamma, E_0, \|\varrho_0^{1/(q+2)}v_0\|_{L^{q+2}}, \|\varrho_0\|_{L^\infty}, \|v_0\|_{L^\infty}\), and \(\|\tau_0\|_{L^\infty}\) such that
		\begin{align*}
			\Theta_{T_0}\leq C,\quad \Theta_{2T_0}\leq C, \dots,\Theta_{jT_0}\leq C,\quad \Theta_{T}\leq C,\quad \forall T\in (j T_0, T^*).
		\end{align*}
		
		Therefore, the proof of Proposition \ref{Prop A_T} is complete.
	\end{proof}

	\section{High-order estimates (3D)}
	After obtaining both the upper bound and the positive lower bound for the density, we can use the parabolic structure of system \(\eqref{Equ2}\) to derive estimates for the higher-order derivatives of \((\varrho, v)\), which then yield estimates for the higher-order derivatives of \((\varrho, u)\) in the original system \(\eqref{Equ1}\). 
	\begin{prop}\label{Prop H1}
		There exists a constant \(C>0\) depending on \(T^*,\gamma, \norm{v_0}_{L^\infty}\),\(\norm{\tau_0}_{L^\infty},\norm{\varrho_0}_{H^2}\) and \(\norm{v_0}_{H^1}\) such that 
		\begin{align}\label{h1}
			\sup_{0\leq t\leq T}\Big(\norm{\varrho}_{H^2}+\norm{\varrho_t}_{L^2}+\norm{v}_{H^1}\Big)+\int_0^T\Big(\norm{\varrho}_{H^3}^2+\norm{\varrho_t}_{H^1}^2+\norm{v}_{H^2}^2+\norm{v_t}_{L^2}^2\Big)dt\leq C.
		\end{align}
	\end{prop}
	\begin{proof}
		Since $\varrho>0$, we can rewrite system \(\eqref{Equ2}\) as:
			\begin{equation}
			\label{Equ3}
			\left\{
			\begin{array}{l}
				\varrho_t+\div(\varrho v)-\Delta \varrho=0,\\
				v_t+(v-2\nabla\log\varrho)\cdot\nabla v-\Delta v+\frac{\gamma}{\gamma-1}\nabla\varrho^{\gamma-1}=0.
			\end{array}
			\right.
		\end{equation}
		Multiplying \(\eqref{Equ3}_1\) by \(\Delta \varrho_t\), integrating the resulting equation over \(\mathbb{T}^3\), and using integration by parts as well as Young’s inequality, we obtain
			\begin{align}\label{hh1}
			\begin{split}
				&\quad\frac{1}{2}\frac{d}{dt}\int |\Delta \varrho|^2dx+\int |\nabla\varrho_t|^2dx\\
				&\leq \int |\nabla\varrho_t||\nabla\divg(\varrho v)|dx\\
				&\leq \frac{1}{4}\int |\nabla\varrho_t|^2dx+C\int |\nabla\divg(\varrho v)|^2dx.
			\end{split}
			\end{align}
		Applying \(\nabla\) to \(\eqref{Equ3}_1\) yields
			\begin{align}\label{hh2}
			\begin{split}
				\frac{1}{8}\int |\nabla\Delta \varrho|^2dx\leq \frac{1}{4}\int |\nabla\varrho_t|^2dx+\frac{1}{4}\int |\nabla\divg(\varrho v)|^2dx.
			\end{split}
		\end{align}
		Adding \(\eqref{hh1}\) and \(\eqref{hh2}\) gives
			\begin{align}\label{hh3}
			\begin{split}
				&\quad\frac{1}{2}\frac{d}{dt}\int |\Delta \varrho|^2dx+\frac{1}{2}\int |\nabla\varrho_t|^2dx+\frac{1}{8}\int |\nabla\Delta \varrho|^2dx\\
				&\leq C\int |\nabla\divg(\varrho v)|^2dx\\
				&\leq C\int |\nabla^2\varrho|^2|v|^2dx+C\int|\nabla\varrho|^2|\nabla v|^2dx+C\int \varrho^2|\nabla^2 v|^2dx\\
				&\leq C\int |\Delta\varrho|^2dx+C\int|\nabla\varrho|^2|\nabla v|^2dx+C_8\int|\Delta v|^2dx,
			\end{split}
			\end{align}
		where the last inequality follows from \(\eqref{Cv}\), \(\eqref{3d RT}\), and \(\eqref{3d A_T}\).
		
		Next, we multiply \(\eqref{Equ3}_2\) by \(\Delta v\), integrate the resulting equation over \(\mathbb{T}^3\), and apply integration by parts together with \(\eqref{Cv}\), \(\eqref{3d RT}\), \(\eqref{3d A_T}\), and \(\eqref{ene est}\) to obtain
		\begin{align}\label{hh4}
			\begin{split}
				&\quad\frac{1}{2}\frac{d}{dt}\int |\nabla v|^2dx+\int |\Delta v|^2dx\\
				&\leq C\int |\Delta v|\Big(|v||\nabla v|+|\nabla\log\varrho||\nabla v|+|\nabla\varrho^{\gamma-1}|\Big)dx\\
				&\leq \frac{1}{2}\int |\Delta v|^2dx+C\int |v|^2|\nabla v|^2dx+C\int |\nabla\varrho|^2|\nabla v|^2dx+C\int |\nabla\varrho|^2dx\\
				&\leq \frac{1}{2}\int |\Delta v|^2dx+C\int |\nabla v|^2dx+C\int |\nabla\varrho|^2|\nabla v|^2dx+C.
			\end{split}
		\end{align}
		Multiplying \(\eqref{hh4}\) by \(4C_8\) and adding the resulting equation to \(\eqref{hh3}\), we obtain
		\begin{align}\label{hh4.5}
			\begin{split}
				&\quad\frac{1}{2}\frac{d}{dt}\int |\Delta \varrho|^2dx+2C_8\frac{d}{dt}\int |\nabla v|^2dx+\frac{1}{2}\int |\nabla\varrho_t|^2dx+\frac{1}{8}\int |\nabla\Delta \varrho|^2dx+C_8\int |\Delta v|^2dx\\
				&\leq C\int |\Delta\varrho|^2dx+C\int|\nabla\varrho|^2|\nabla v|^2dx+C\int |\nabla v|^2dx+C.
			\end{split}
		\end{align}
		
		Finally, we need to estimate \(C\int |\nabla \varrho|^2 |\nabla v|^2 \, dx\). Integration by parts together with Young's inequality yields
			\begin{align}\label{hh5}
			\begin{split}
				C\int |\nabla\varrho|^2|\nabla v|^2dx&=C\int \partial_i\varrho\partial_i\varrho\partial_j v_k\partial_j v_kdx\\
				&=-2C\int \partial_{ij}\varrho\partial_i\varrho v_k\partial_j v_kdx-C\int\partial_i\varrho\partial_i\varrho v_k\partial_{jj}v_kdx\\
				&\leq C\int |\nabla^2\varrho||\nabla\varrho||v||\nabla v|dx+C\int |\nabla\varrho|^2|v||\nabla^2 v|dx.
			\end{split}
		\end{align}
		Using \(\eqref{Cv}\),\(\eqref{3d RT}\), \(\eqref{3d A_T}\), and \(\eqref{ene est}\), we obtain
		\begin{align}\label{hh6}
			\begin{split}
				C\int |\nabla^2\varrho||\nabla\varrho||v||\nabla v|dx&\leq C\norm{\nabla^2\varrho}_{L^6}\norm{\nabla\varrho}_{L^2}\norm{v}_{L^\infty}\norm{\nabla v}_{L^3}\\
				&\leq \frac{1}{32}\norm{\nabla^2\varrho}_{H^1}^2+C\norm{\nabla v}_{L^3}^2\\
				&\leq \frac{1}{32}\norm{\nabla^2\varrho}_{H^1}^2+C\norm{\nabla v}_{L^2}\norm{\nabla v}_{L^6}\\
				&\leq \frac{1}{32}\norm{\nabla^2\varrho}_{H^1}^2+\frac{C_8}{4}\norm{\nabla v}_{H^1}^2+C\norm{\nabla v}_{L^2}^2.
			\end{split}
		\end{align}
		Similarly,
		\begin{align}\label{hh7}
			\begin{split}
				C\int |\nabla\varrho|^2|v||\nabla^2 v|dx&\leq C\norm{\nabla\varrho}_{L^4}^2\norm{v}_{L^\infty}\norm{\nabla^2 v}_{L^2}\\
				&\leq \frac{C_8}{4}\norm{\nabla^2 v}_{L^2}^2+C\norm{\nabla\varrho}_{L^4}^4\\
				&\leq\frac{C_8}{4}\norm{\nabla^2 v}_{L^2}^2+C\int \varrho|\nabla\varrho|^2|\nabla^2 \varrho| dx\\
				&\leq \frac{C_8}{4}\norm{\nabla^2 v}_{L^2}^2+C\norm{\varrho}_{L^\infty}\norm{\nabla\varrho}_{L^{\frac{12}{5}}}^2\norm{\nabla^2\varrho}_{L^6}\\
				&\leq  \frac{C_8}{4}\norm{\nabla^2v}_{L^2}^2+\frac{1}{32}\norm{\nabla^2 \varrho}_{H^1}^2+C\norm{\nabla\varrho}_{L^{\frac{12}{5}}}^4\\
				&\leq  \frac{C_8}{4}\norm{\nabla^2v}_{L^2}^2+\frac{1}{32}\norm{\nabla^2 \varrho}_{H^1}^2+C\norm{\nabla\varrho}_{L^{2}}^3\norm{\nabla\varrho}_{L^6}\\
				&\leq  \frac{C_8}{4}\norm{\nabla^2v}_{L^2}^2+\frac{1}{32}\norm{\nabla^2 \varrho}_{H^1}^2+C\norm{\nabla\varrho}_{H^1}\\
				&\leq  \frac{C_8}{4}\norm{\nabla^2v}_{L^2}^2+\frac{1}{32}\norm{\nabla^2 \varrho}_{H^1}^2+C\norm{\Delta \varrho}_{L^2}^2+C,
			\end{split}
		\end{align}
        where in the third inequality we have used integration by parts. Substituting \(\eqref{hh5}\)–\(\eqref{hh7}\) into \(\eqref{hh4.5}\) yields
			\begin{align*}
			\begin{split}
				&\quad\frac{1}{2}\frac{d}{dt}\int |\Delta \varrho|^2dx+2C_8\frac{d}{dt}\int |\nabla v|^2dx+\frac{1}{2}\int |\nabla\varrho_t|^2dx+\frac{1}{16}\int |\nabla\Delta \varrho|^2dx+\frac{C_8}{2}\int |\Delta v|^2dx\\
				&\leq C\int |\Delta\varrho|^2dx+C\int |\nabla v|^2dx+C.
			\end{split}
		\end{align*}
		 Therefore, applying Gr\"onwall’s inequality together with \(\eqref{ene est}\), \(\eqref{ent est}\), $\eqref{3d RT}$, and \(\eqref{3d A_T}\), we obtain
		 	\begin{align}\label{hh8}
		 	\sup_{0\leq t\leq T}\Big(\norm{\varrho}_{H^2}+\norm{v}_{H^1}\Big)+\int_0^T\Big(\norm{\varrho}_{H^3}^2+\norm{\varrho_t}_{H^1}^2+\norm{v}_{H^2}^2\Big)dt\leq C.
			 \end{align}
		 Moreover, from \(\eqref{Equ3}_1,\eqref{hh8},\) $\eqref{Cv}$, and $\eqref{3d A_T}$ we have
		 \begin{align}\label{hh9}
		 	\begin{split}
		 		\norm{\varrho_t}_{L^2}&\leq C\norm{\Delta\varrho}_{L^2}+C\norm{\nabla(\varrho v)}_{L^2}\\
		 		&\leq C\norm{\Delta\varrho}_{L^2}+C\norm{\varrho}_{L^\infty}\norm{\nabla v}_{L^2}+C\norm{\nabla\varrho}_{L^2}\norm{v}_{L^\infty}\\&\leq C.
		 	\end{split}
		 \end{align}
		 Similarly, 
		 \begin{align}\label{hh10}
		 	\begin{split}
		 		\norm{v_t}_{L^2}&\leq C\norm{|v||\nabla v|}_{L^2}+C\norm{|\nabla\log\varrho| |\nabla v|}_{L^2}+C\norm{\Delta v}_{L^2}+C\norm{\nabla\varrho^{\gamma-1}}_{L^2}\\
		 		&\leq C\norm{v}_{L^\infty}\norm{\nabla v}_{L^2}+C\norm{\nabla\varrho}_{L^6}\norm{\nabla v}_{L^3}+C\norm{\Delta v}_{L^2}+C\norm{\nabla \varrho}_{L^2}\\
		 		&\leq C+C\norm{\Delta v}_{L^2}.
		 	\end{split}
		 \end{align}
		Combining \(\eqref{hh8}\)–\(\eqref{hh10}\) proves \(\eqref{h1}\). This completes the proof of Proposition \ref{Prop H1}.
	\end{proof}
	
	\begin{prop}\label{Prop H2}
		There exists a constant \(C>0\) depending on \(T^*,\gamma \),\(\norm{\tau_0}_{L^\infty},\norm{\varrho_0}_{H^2}\) and \(\norm{v_0}_{H^2}\) such that 
		\begin{align}\label{h2}
			\begin{split}
				\sup_{0\leq t\leq T}\Big(\norm{v}_{H^2}+\norm{v_t}_{L^2}\Big)
				+\int_0^T\Big(\norm{v}_{H^3}^2+\norm{v_t}_{H^1}^2\Big)dt\leq C.
			\end{split}
		\end{align}
	\end{prop}
	\begin{proof}
		Differentiating \(\eqref{Equ3}_2\) with respect to \(t\) gives
		\begin{align*}
			v_{tt}+(v-2\nabla\log\varrho)_t\cdot\nabla v+(v-2\nabla\log\varrho)\cdot \nabla v_t-\Delta v_t+\frac{\gamma}{\gamma-1}(\nabla \varrho^{\gamma-1})_t=0.
		\end{align*}
		Multiplying the above equation by \(v_t\), integrating the resulting equation over \(\mathbb{T}^3\), and using integration by parts as well as Young’s inequality, we obtain
		\begin{align}\label{hh16}
			\begin{split}
				&\quad\frac{1}{2}\frac{d}{dt}\int |v_t|^2dx+\int |\nabla v_t|^2dx\\
				&=-\int (v-2\nabla\log\varrho)_t\cdot\nabla v\cdot v_tdx-\int (v-2\nabla\log\varrho)\cdot \nabla v_t\cdot v_tdx\\
				&\quad-\frac{\gamma}{\gamma-1}\int \nabla(\varrho^{\gamma-1})_t\cdot v_tdx\\
				&=\int \div v_tv\cdot v_tdx+\int v_t\cdot\nabla v_t\cdot vdx-2\int(\log\varrho)_t\Delta v\cdot v_tdx-2\int (\log\varrho)_t\nabla v:\nabla v_tdx\\
				&\quad-\int (v-2\nabla\log\varrho)\cdot \nabla v_t\cdot v_tdx+\frac{\gamma}{\gamma-1}\int (\varrho^{\gamma-1})_t\divg v_tdx\\
				&\leq \frac{1}{4}\int |\nabla v_t|^2dx+C\int \Big(|v|^2|v_t|^2+|\varrho_t|^2|\nabla v|^2+|\nabla\varrho|^2|v_t|^2+|\varrho_t|^2\Big)dx+C\int |\varrho_t||\Delta v||v_t|dx.
			\end{split}
		\end{align}
		It follows from $\eqref{Equ3}_2$ that
		\begin{align}\label{hh17}
			\begin{split}
				\quad\frac{1}{12}\int |\nabla\Delta v|^2dx&=\frac{1}{12}\int |\nabla(v_t+(v-2\nabla\log\varrho)\cdot\nabla v+\frac{\gamma}{\gamma-1}\nabla\varrho^{\gamma-1})|^2dx\\
				&\leq \frac{1}{4}\int |\nabla v_t|^2dx+C\int \Big(|\nabla v|^4+|v|^2|\nabla^2 v|^2+|\nabla\varrho|^4|\nabla v|^2\\
				&\quad+|\nabla^2\varrho|^2|\nabla v|^2+|\nabla\varrho|^2|\nabla^2 v|^2+|\nabla\varrho|^4+|\nabla^2\varrho|^2\Big)dx.
			\end{split}
		\end{align}
		Adding \(\eqref{hh16}\) and \(\eqref{hh17}\), we obtain from $\eqref{h1}$, $\eqref{Cv}$, and \eqref{3d A_T} that
		\begin{align}
			\begin{split}
				&\quad\frac{1}{2}\frac{d}{dt}\int |v_t|^2dx+\frac{1}{2}\int |\nabla v_t|^2dx+\frac{1}{12}\int |\nabla\Delta v|^2dx\\
				&\leq C\int \Big(|v|^2|v_t|^2+|\varrho_t|^2|\nabla v|^2+|\nabla\varrho|^2|v_t|^2+|\varrho_t|^2\Big)dx+C\int |\varrho_t||\Delta v||v_t|dx\\
				&\quad +C\int \Big(|\nabla v|^4+|v|^2|\nabla^2 v|^2+|\nabla\varrho|^4|\nabla v|^2+|\nabla^2\varrho|^2|\nabla v|^2+|\nabla\varrho|^2|\nabla^2 v|^2+|\nabla\varrho|^4+|\nabla^2\varrho|^2\Big)dx\\
				&\leq C\Big(\norm{v}_{L^\infty}^2\norm{v_t}_{L^2}^2+\norm{\varrho_t}_{L^2}^2\norm{\nabla v}_{L^\infty}^2+\norm{\nabla\varrho}_{L^\infty}^2\norm{v_t}_{L^2}^2+\norm{\varrho_t}_{L^2}^2\Big)+C\norm{\varrho_t}_{L^2}\norm{\Delta v}_{L^3}\norm{v_t}_{L^6}\\
				&\quad+C\Big(\norm{v}_{L^\infty}\norm{\nabla v}_{L^{\frac{12}{5}}}^2\norm{\nabla^2 v}_{L^6}+\norm{v}_{L^\infty}^2\norm{\nabla^2 v}_{L^2}^2+\norm{\nabla\varrho}_{L^4}^4\norm{\nabla v}_{L^\infty}^2+\norm{\nabla^2\varrho}_{L^2}^2\norm{\nabla v}_{L^\infty}^2\\
				&\quad+\norm{\nabla \varrho}_{L^6}^2\norm{\nabla^2 v}_{L^3}^2+\norm{\nabla\varrho}_{L^4}^4+\norm{\nabla^2\varrho}_{L^2}^2\Big)\\
				&\leq C\Big(\norm{v_t}_{L^2}^2+\norm{\nabla v}_{L^2}^{\frac{1}{2}}\norm{\nabla v}_{W^{1,6}}^{\frac{3}{2}}+\norm{\nabla\varrho}_{H^2}^2\norm{v_t}_{L^2}^2+1\Big)+C\norm{\nabla^2 v}_{L^2}^{\frac{1}{2}}\norm{\nabla^2 v}_{L^6}^{\frac{1}{2}}\norm{v_t}_{H^1}\\
				&\quad+C\Big(\norm{\nabla v}_{L^2}^{\frac{3}{2}}\norm{\nabla v}_{L^6}^{\frac{1}{2}}\norm{\nabla^2 v}_{L^6}+\norm{\nabla^2 v}_{L^2}^2+\norm{\nabla v}_{L^2}^{\frac{1}{2}}\norm{\nabla v}_{W^{1,6}}^{\frac{3}{2}}+\norm{\nabla^2 v}_{L^2}\norm{\nabla^2 v}_{L^6}+1\Big)\\
				&\leq C\Big(\norm{v_t}_{L^2}^2+\norm{\nabla v}_{L^2}^{\frac{1}{2}}\norm{\nabla v}_{H^2}^{\frac{3}{2}}+\norm{\nabla\varrho}_{H^2}^2\norm{v_t}_{L^2}^2+1\Big)+C\norm{\nabla^2 v}_{L^2}^{\frac{1}{2}}\norm{\nabla^2 v}_{H^1}^{\frac{1}{2}}\norm{v_t}_{H^1}\\
				&\quad+C\Big(\norm{\nabla v}_{L^2}^{\frac{3}{2}}\norm{\nabla v}_{H^1}^{\frac{1}{2}}\norm{\nabla^2 v}_{H^1}+\norm{\nabla^2 v}_{L^2}^2+\norm{\nabla v}_{L^2}^{\frac{1}{2}}\norm{\nabla v}_{H^2}^{\frac{3}{2}}+\norm{\nabla^2 v}_{L^2}\norm{\nabla^2 v}_{H^1}+1\Big)\\
				&\leq \frac{1}{4}\norm{v_t}_{H^1}^2+\frac{1}{24}\norm{\nabla v}_{H^2}^2+C\Big(1+\norm{\nabla\varrho}_{H^2}^2\Big)\norm{v_t}_{L^2}^2+C\Big(1+\norm{\nabla v}_{L^2}^2+\norm{\nabla^2 v}_{L^2}^2\Big),
			\end{split}
		\end{align}
		where in the second inequality we have used integration by parts. Consequently,
		\begin{align}
			\begin{split}
				&\quad\frac{1}{2}\frac{d}{dt}\int |v_t|^2dx+\frac{1}{4}\int |\nabla v_t|^2dx+\frac{1}{24}\int |\nabla\Delta v|^2dx\\
				&\leq C\Big(1+\norm{\nabla\varrho}_{H^2}^2\Big)\norm{v_t}_{L^2}^2+C\Big(1+\norm{\nabla v}_{L^2}^2+\norm{\nabla^2 v}_{L^2}^2\Big).
			\end{split}
		\end{align}
		Applying Gr\"onwall's inequality together with \(\eqref{h1}\) yields
		\begin{align}\label{hh17.5}
			\begin{split}
				\sup_{0\leq t\leq T}\norm{v_t}_{L^2}+\int_0^T\Big(\norm{v_t}_{H^1}^2+\norm{v}_{H^3}^2\Big)dt\leq C.
			\end{split}
		\end{align}
		Moreover, from \(\eqref{Equ3}_2,\eqref{Cv},\eqref{3d A_T},\eqref{h1}\), and $\eqref{hh17.5}$ we know that
		\begin{align}
			\begin{split}
				\norm{\nabla^2 v}_{L^2}&\leq C\Big(\norm{v_t}_{L^2}+\norm{v}_{L^\infty}\norm{\nabla v}_{L^2}+\norm{\nabla\log\varrho}_{L^6}\norm{\nabla v}_{L^3}+\norm{\nabla\varrho}_{L^2}\Big)\\
				&\leq C+C\norm{\nabla v}_{L^2}^{\frac{1}{2}}\norm{\nabla v}_{H^1}^{\frac{1}{2}}\\
				&\leq \frac{1}{2}\norm{\nabla^2 v}_{L^2}+C,
			\end{split}
		\end{align}
		which implies that
		\begin{align}\label{hh17.6}
			\sup_{0\leq t\leq T}\norm{\nabla^2v}_{L^2}\leq C.
		\end{align}
		Combining \eqref{h1}, \eqref{hh17.5}, and \eqref{hh17.6} completes the proof of Proposition \ref{Prop H2}.
	\end{proof}
	
	\begin{prop}\label{Prop H3}
		There exists a constant \(C>0\) depending on \(T^*,\gamma \),\(\norm{\tau_0}_{L^\infty},\norm{\varrho_0}_{H^3}\) and \(\norm{v_0}_{H^2}\) such that 
		\begin{align}\label{h3}
			\begin{split}
				\sup_{0\leq t\leq T}\Big(\norm{\varrho}_{H^3}+\norm{\varrho_t}_{H^1}\Big)
				+\int_0^T\Big(\norm{\varrho}_{H^4}^2+\norm{\varrho_t}_{H^2}^2+\norm{\varrho_{tt}}_{L^2}^2\Big)dt\leq C.
			\end{split}
		\end{align}
	\end{prop}
	\begin{proof}
		Differentiating equation \(\eqref{Equ3}_1\) with respect to $t$ gives:
		\begin{align}\label{hh10.5}
			\varrho_{tt}+(\div (\varrho v))_t-\Delta\varrho_t=0.
		\end{align}
		Multiplying the above equation by \(\varrho_{tt}\), integrating the resulting equation over \(\mathbb{T}^3\), and applying integration by parts as well as Young’s inequality yields
		\begin{align}\label{hh11}
			\begin{split}
				&\quad\frac{1}{2}\frac{d}{dt}\int |\nabla\varrho_t|^2dx+\int |\varrho_{tt}|^2dx\\
				&\leq \int |\varrho_{tt}||(\divg(\varrho v))_t|dx\\
				&\leq \frac{1}{4}\int |\varrho_{tt}|^2dx+C\int|(\divg(\varrho v))_t|^2dx.
			\end{split}
		\end{align}
		It follows from $\eqref{hh10.5}$ that
		\begin{align}\label{hh12}
			\begin{split}
				\frac{1}{8}\int |\Delta\varrho_t|^2dx&\leq \frac{1}{4}\int |\varrho_{tt}|^2+\frac{1}{4}\int |(\divg(\varrho v))_t|^2dx.
			\end{split}
		\end{align}
		Applying the Laplacian \(\Delta\) to equation \(\eqref{Equ3}_1\) gives:
		\begin{align}\label{hh13}
			\Delta\varrho_t+\Delta\divg(\varrho v)-\Delta\Delta\varrho=0,
		\end{align}
		which implies that
		\begin{align}\label{hh14}
			\begin{split}
				\frac{1}{32}\int|\Delta\Delta \varrho|^2dx\leq \frac{1}{16}\int |\Delta\varrho_t|^2dx+\frac{1}{16}\int |\Delta\divg(\varrho v)|^2dx.
			\end{split}
		\end{align}
		Adding \(\eqref{hh11}\), \(\eqref{hh12}\), and \(\eqref{hh14}\) gives
		\begin{align}\label{hh15}
			\begin{split}
				&\quad\frac{1}{2}\frac{d}{dt}\int |\nabla\varrho_t|^2dx+\frac{1}{2}\int |\varrho_{tt}|^2dx+\frac{1}{16}\int |\Delta\varrho_t|^2dx+\frac{1}{32}\int |\Delta\Delta\varrho|^2dx\\
				&\leq C\int|(\divg(\varrho v))_t|^2dx+C\int |\Delta\divg(\varrho v)|^2dx\\
				&\leq C\int \Big(|\varrho_t|^2|\nabla v|^2+|\varrho|^2|\nabla v_t|^2+|\nabla\varrho_t|^2|v|^2+|\nabla\varrho|^2|v_t|^2\Big)dx\\
				&\quad+C\int \Big(|\nabla^3\varrho|^2|v|^2+|\nabla^2\varrho|^2|\nabla v|^2+|\nabla\varrho|^2|\nabla^2 v|^2+|\varrho|^2|\nabla^3 v|^2\Big)dx\\
				&\leq C\Big(\norm{\varrho_t}_{L^2}^2\norm{\nabla v}_{L^\infty}^2+\norm{\varrho}_{L^\infty}^2\norm{\nabla v_t}_{L^2}^2+\norm{\nabla\varrho_t}_{L^2}^2\norm{v}_{L^\infty}^2+\norm{\nabla\varrho}_{L^\infty}^2\norm{v_t}_{L^2}^2\Big)\\
				&\quad+C\Big(\norm{\nabla^3\varrho}_{L^2}^2\norm{v}_{L^\infty}^2+\norm{\nabla^2\varrho}_{L^2}^2\norm{\nabla v}_{L^\infty}^2+\norm{\nabla\varrho}_{L^6}^2\norm{\nabla^2v}_{L^3}^2+\norm{\varrho}_{L^\infty}^2\norm{\nabla^3 v}_{L^2}^2\Big)\\
				&\leq C\Big(\norm{\nabla v}_{H^2}^2+\norm{\nabla v_t}_{L^2}^2+\norm{\nabla\varrho_t}_{L^2}^2+\norm{\nabla\varrho}_{H^2}^2\Big)\\
				&\quad+C\Big(\norm{\nabla^3\varrho}_{L^2}^2+\norm{\nabla v}_{H^2}^2+\norm{\nabla^2 v}_{H^1}^2+\norm{\nabla^3 v}_{L^2}^2\Big)\\
				&\leq C\Big(\norm{\nabla v}_{H^2}^2+\norm{\nabla v_t}_{L^2}^2+\norm{\nabla\varrho_t}_{L^2}^2+\norm{\nabla\varrho}_{H^2}^2\Big),
			\end{split}
		\end{align}
		where we have used \eqref{h1} and $\eqref{h2}$. Integrating \eqref{hh15} with respect to time and using \eqref{h1} and $\eqref{h2}$, we obtain
		\begin{align}\label{hh19}
			\begin{split}
				\sup_{0\leq t\leq T}\norm{\nabla\varrho_t}_{L^2}+\int_0^T\Big(\norm{\varrho_{tt}}_{L^2}^2+\norm{\varrho_t}_{H^2}^2+\norm{\varrho}_{H^4}^2\Big)dt\leq C.
			\end{split}
		\end{align}
		Moreover, from \(\eqref{Equ3}_1\), $\eqref{hh19}$, $\eqref{h1}$, and $\eqref{h2}$ we have that
		\begin{align}
			\begin{split}
				\norm{\nabla\Delta\varrho}_{L^2}&\leq C\Big(\norm{\nabla\varrho_t}_{L^2}+\norm{\nabla^2\varrho}_{L^2}\norm{v}_{L^\infty}+\norm{\nabla\varrho}_{L^6}\norm{\nabla v}_{L^3}+\norm{\varrho}_{L^\infty}\norm{\nabla^2 v}_{L^2}\Big)\\
				&\leq C,
			\end{split}
		\end{align}
		which implies that
		\begin{align}\label{hh20}
			\sup_{0\leq t\leq T}\norm{\nabla^3 \varrho}_{L^2}\leq C.
		\end{align}
		Combining \eqref{h1}, \eqref{hh19}, and \eqref{hh20} completes the proof of Proposition \ref{Prop H3}.
	\end{proof}

    \section{Proof of Theorem \ref{Thm 1.1}}
    In the three-dimensional case, the global existence of strong solutions to problem \eqref{Equ1}--\eqref{ini data} follows from the a priori estimates established in the preceding sections. The proof is standard and is therefore omitted. For the uniqueness of the strong solution, we refer to \cite{Yu-Wu-2021}. 
    
    In the two-dimensional case, the global existence of strong solutions to problem \eqref{Equ1}--\eqref{ini data} can be established for any $\gamma \ge 1$. In fact, the restriction on $\gamma$ in the three-dimensional case stems primarily from the need to estimate an upper bound for the density. In 2D, the standard energy estimate \eqref{ene est} ensures that $\norm{\varrho}_{L^\infty_T L^p}$ is bounded for any $p\in[1,\infty)$. This, in turn, yields the boundedness of $\norm{\varrho^{\frac{1}{4}} v}_{L^\infty_T L^{4}}$ for every $\gamma\ge1$. With this bound at hand, the De Giorgi iteration technique can be applied to obtain an upper bound for the density. The complete implementation of this step is presented in \cite{Yu-Wu-2021}.

	\section*{Acknowledgments}
	X Huang is partially supported by Chinese Academy of Sciences Project for Young Scientists in Basic Research (Grant No. YSBR-031), National Natural Science Foundation of China (Grant Nos. 12494542, 11688101) and National Key R\&D Program of China (Grant No. 2021YFA1000801). X Zhang is supported by National Natural Science Foundation of Yulin
	University (Grant No. 2025GK11), the Young Elite Scientists Sponsorship Program by Yulin Association for Science and Technology (Grant No. 20250712).
	
	\vspace{1cm}
	\noindent\textbf{Data availability statement.} Data sharing is not applicable to this article.
	
	\vspace{0.3cm}
	\noindent\textbf{Conflict of interest.} The authors declare that they have no conflict of interest.

\end{document}